\documentclass{amsart}

\usepackage{amsmath}
\usepackage{amsthm}
\usepackage{hyperref}
\usepackage{amsfonts,graphics,amsthm,amsfonts,amscd,latexsym}
\usepackage{epsfig}
\usepackage{flafter}
\usepackage{mathtools}
\usepackage{comment}
\usepackage{stmaryrd}

\usepackage{mathabx,epsfig}

\hypersetup{
    colorlinks=true,    
    linkcolor=blue,          
    citecolor=blue,      
    filecolor=blue,      
    urlcolor=blue           
}
\usepackage{tikz}
\usetikzlibrary{graphs,positioning,arrows,shapes.misc,decorations.pathmorphing}

\tikzset{
    >=stealth,
    every picture/.style={thick},
    graphs/every graph/.style={empty nodes},
}

\tikzstyle{vertex}=[
    draw,
    circle,
    fill=black,
    inner sep=1pt,
    minimum width=5pt,
]
\usepackage[position=top]{subfig}
\usepackage{amssymb}
\usepackage{color}

\calclayout

\usetikzlibrary{decorations.pathmorphing}
\tikzstyle{printersafe}=[decoration={snake,amplitude=0pt}]

\newcommand{\rank}{\operatorname{rank}}

\newcommand{\pp}{\mathbb{P}}

\newcommand{\qq}{\mathbb{Q}}
\newcommand{\zz}{\mathbb{Z}}

\newcommand{\rr}{\mathbb{R}}

\newcommand{\kk}{\mathbb{K}}

\def\O#1.{\mathcal {O}_{#1}}			
\def\pr #1.{\mathbb P^{#1}}				
\def\af #1.{\mathbb A^{#1}}			
\def\ses#1.#2.#3.{0\to #1\to #2\to #3 \to 0}	
\def\xrar#1.{\xrightarrow{#1}}			
\def\K#1.{K_{#1}}						
\def\bA#1.{\mathbf{A}_{#1}}			
\def\bM#1.{\mathbf{M}_{#1}}				
\def\bL#1.{\mathbf{L}_{#1}}				
\def\bB#1.{\mathbf{B}_{#1}}				
\def\bK#1.{\mathbf{K}_{#1}}			
\def\subs#1.{_{#1}}					
\def\sups#1.{^{#1}}

\DeclareMathOperator{\reg}{reg}

\usepackage{tikz}
\usetikzlibrary{matrix,arrows,decorations.pathmorphing}

\newtheorem{introthm}{Theorem}

  \newtheorem{theorem}{Theorem}[section]
  \newtheorem{lemma}[theorem]{Lemma}
  \newtheorem{proposition}[theorem]{Proposition}

  \newtheorem{notation}[theorem]{Notation}
  \newtheorem{definition}[theorem]{Definition}
  \newtheorem{example}[theorem]{Example}

  \newtheorem{question}[theorem]{Question}

\newtheorem{remark}[theorem]{Remark}

\theoremstyle{remark}

\numberwithin{equation}{section}

\usepackage[all]{xy}

\begin{document}

\title[Lct's and coregularity]{Log canonical thresholds and coregularity}

\author[F. Figueroa]{Fernando Figueroa}
\address{Department of Mathematics, Princeton University, Fine Hall, Washington Road, Princeton, NJ 08544-1000, USA}
\email{fzamora@princeton.edu}

\author[J.~Moraga]{Joaqu\'in Moraga}
\address{Department of Mathematics, Princeton University, Fine Hall, Washington Road, Princeton, NJ 08544-1000, USA
}
\email{jmoraga@princeton.edu}

\author[J.~Peng]{Junyao Peng}
\address{Department of Mathematics, Princeton University, Fine Hall, Washington Road, Princeton, NJ 08544-1000, USA
}
\email{junyaop@princeton.edu}

\maketitle

\begin{abstract}
We prove the ascending chain condition for log canonical thresholds of bounded coregularity.
\end{abstract}

\setcounter{tocdepth}{1} 
\tableofcontents

\section{Introduction}

The main approach to study algebraic singularities is to introduce invariants that allow us to measure how singular a point in an algebraic variety is.
Once the invariant is introduced, the next step is to understand what values it can take, 
whether it detects smoothness, how the numerical data reflects on the singularity
and vice-versa.
The best-known invariant of singularities is the multiplicity.
The multiplicity played a fundamental role in the resolution of singularities~\cite{Kol07}.
However, it is often too coarse of an invariant to obtain information in other contexts.
This has led algebraic geometers to search for new invariants of singularities that could guide further developments in the understanding of algebraic singularities.

The log canonical threshold,
formerly known as complex singularity exponent,
was introduced by Atiyah in the study of the division of distributions using resolution of singularities~\cite{Ati70}.
The first properties of log canonical thresholds were proved by Varchenko in connection with the asymptotic expansion of integrals
and Hodge structures~\cite{Var83}.
Shokurov realized that the log canonical threshold could be defined purely in terms of the singularities of the minimal model program~\cite{Sho92}.
They conjectured that log canonical thresholds of the same dimension satisfy the ascending chain condition (ACC), i.e., there is no infinite increasing sequence.
This conjecture follows a common philosophy in algebraic geometry:
in a fixed dimension, we can not find a sequence of singularities that get less and less singular.
The previous principle is not mathematically precise, but
it has led to many important conjectures on the behavior of algebraic singularities (see, e.g.,~\cite{Sho04,BS10}).

Although its modern definition 
is purely in terms 
in terms of log discrepancies, 
the log canonical threshold has deep connections with other topics in mathematics:
the growth of the number of polynomial solutions in $\zz/p^n$~\cite{Igu00},
dimensions of jet schemes~\cite{Mus02},
the theory of D-modules~\cite{Kol13},
the $\alpha$-invariant and
algebraic K-stability~\cite{Xu20},
vanishing theorems~\cite{Laz04b}, non-rationality of Fano varieties~\cite{dFEM03a},
and positive characteristic methods~\cite{BHMM12}.
It is also worth mentioning that the log canonical threshold is deeply related to the termination of flips~\cite{Bir12,Mor18a,HM20}.

We briefly recall some of the developments toward the understanding of log canonical thresholds.
In~\cite{Kuw98,Kuw99a,Kuw99b}, Kuwata computed log canonical thresholds of hypersurface singularities,
surfaces in $\mathbb{C}^3$, 
and reducible plane curves.
In~\cite{Pro01,Pro02}, Prokhorov started the study of accumulation points of log canonical thresholds, conjecturing that they must come from log canonical thresholds in lower dimensions.
De Fernex, Ein, and Musta\c{t}\u{a} drew a connection between the log canonical threshold and multiplicity in~\cite{dFEM03b}.
In~\cite{Chel09}, Cheltsov studied the log canonical threshold of Fano threefold hypersurfaces and its relation to birational rigidity.
In~\cite{dFEM10}, De Fernex, Ein, and Musta\c{t}\u{a} proved the ascending chain condition for log canonical thresholds in smooth varieties.
In that paper, the authors used an inversion of adjunction for complete intersection varieties~\cite{EM04}.
Then, they achieved the ACC for log canonical thresholds on varieties with bounded singularities in~\cite{dFEM11}.
This achievement used the theory of ultrafilters.
In~\cite{HMX14}, Hacon, McKernan, and Xu, proved the ACC for log canonical thresholds. 
Using ideas from the minimal model program, 
the authors reduced this statement to a problem about projective varieties of log Calabi--Yau type. 
Then, they improved the Hacon-McKernan developments on birational boundedness of varieties of log general type to prove this global result.
In~\cite{BZ16}, Birkar and Zhang generalized the results for log canonical thresholds to the context of generalized pairs. This improvement was vital for the effectivity of Iitaka fibrations and boundedness of Fano varieties~\cite{BZ16,Bir19,Bir21}.
In~\cite{McL19}, McLean proved that the log canonical threshold can be expressed in terms of Floer cohomology.

Now, we turn to explain the main theorem of this article.
When computing log canonical thresholds, it is natural to start with orbifold and toric examples.
If $(T;t)$ is the germ of a toric singularity and $\Gamma$ is a reduced torus invariant divisor, then the pair 
$(T,\Gamma;t)$ is always log canonical.
This means that, regardless of the dimension of the germ, 
the log canonical threshold is only one.
This leads to a natural question:
are log canonical thresholds controlled by the dimension
or can they be controlled by a weaker invariant?
Some computations show that the closer to the toric setting we are, the fewer values for log canonical thresholds we can produce (see Example~\ref{ex:1}). 
In order to make this notion precise, we define the {\it coregularity} of a log canonical singularity $(X,B;x)$ to be:
\begin{equation}
\tag{1}
\label{def:coreg-intro}
{\rm coreg}(X,B;x):= \dim X -\dim \mathcal{D}(X,B;x)-1.
\end{equation} 
Here, $\mathcal{D}(X,B;x)$ is the {\it dual complex} of the singularity, which captures the combinatorial 
complexity of its resolution (see, e.g.,~\cite{dFKX17,KX16}).
The coregularity of a log canonical singularity is an integer between zero and its dimension minus one (see Definition~\ref{def:coreg-general}). Toric singularities with the reduced toric boundary have coregularity zero (see Example~\ref{ex:1}). 
The set of log canonical thresholds of coregularity at most $c$ is defined to be:
\[
{\rm LCT}_c(I,J):=
\left\{
t \mid 
t={\rm lct}(X,B;\Gamma),
{\rm coeff}(B)\in I, 
{\rm coeff}(\Gamma)\in J,\text{ and } 
{\rm coreg}(X,B+t\Gamma)\leq c
\right\}.
\] 
The main theorem of this article is the following:
\begin{introthm}\label{introhm:acc-coreg}
Let $c\geq 1$ be a positive integer.
Let $I$ and $J$ be two sets of nonnegative real numbers
satisfying the descending chain condition.
Then, the set ${\rm LCT}_c(I,J)$ satisfies the
ascending chain condition.
\end{introthm}
Note that we do not fix the dimension of our germs. 
Hence, the previous theorem extends to a great extent the main result of~\cite{HMX14}.
Indeed, the set ${\rm LCT}_c(I,J)$ contains all the germs of dimension at most $c+1$ but it also contain germs of unbounded dimension (see Example~\ref{ex:3}).
In Example~\ref{ex:4}, we show that both the conditions on the coefficients and the coregularity are necessary to obtain the ascending chain condition.

Our definition of coregularity in~\eqref{def:coreg-intro} is only for the local case; we call it the {\it local coregularity}.
In Definition~\ref{def:coreg-general}, we give the definition of the coregularity for a normal quasi-projective log canonical pair.
In a few words, a pair has coregularity $c$ if locally around its minimal log canonical centers it has coregularity $c$.
In the log Calabi--Yau setting, this condition can be checked at any point of the dual complex.
Indeed, the dual complex of a log Calabi--Yau pair is equi-dimensional~\cite{KK10}.

In the minimal model program, 
it is usual to study singularities by constructing birational modifications
and reducing the problem about the singularity 
to a problem about the exceptional locus
of this modification.
This sort of argument is known as 
{\it global-to-local principle}
in algebraic geometry.
Thus, we reduce the proof of the 
ACC for lct's with bounded coregularity
(Theorem~\ref{introhm:acc-coreg})
to the following statement about log Calabi--Yau varieties:

\begin{introthm}\label{introthm:acc-global-coreg}
Let $c$ be a positive integer.
Let $I$ be a set satisfying the descending chain condition (DCC).
There exists a finite subset $I_0\subseteq I$ satisfying the following. 
Let $(X,B)$ be a projective log canonical pair so that:
\begin{enumerate}
    \item the equivalence $K_X+B \equiv 0$ holds; 
    \item the set ${\rm coeff}(B)$ is contained in $I$; and 
    \item the coregularity of $(X,B)$ is at most $c$.
\end{enumerate}
Then, we have that ${\rm coeff}(B)\subseteq I_0$.
\end{introthm}

Statements such as the previous one are known as global ascending chain conditions (or global ACC, for short).
We call the previous theorem the {\it global ACC with bounded coregularity}.
As we will show throughout the proof, 
the advantage of working with log canonical thresholds of small coregularity 
is that we can often reduce these computations to low dimensional computations.
Following this philosophy, we can give a precise description of the log canonical thresholds of coregularity zero and one (see Theorem~\ref{thm:coreg-one}).

\begin{introthm}\label{introthm:coreg-zero-one}
Let $I$ and $J$ be two sets of nonnegative real numbers. 
Then, we have that 
\[
{\rm LCT}_0(I,J):=
\left\{ 
\frac{1-\sum_{k=1}^m n_ki_k}{\sum_{k=1}^n m_kj_k}\geq 0 \mid m,n\in \zz_{\geq 0}, 
n_k,m_k\in \zz_{>0},
i_k\in I, \text{ and }
j_k\in J
\right\}.
\] 
\end{introthm}
Note that in the previous theorem 
the set of log canonical thresholds is described with no assumptions on $I$ or $J$.
However, it is clear from the description
that it satisfies the ascending chain condition
if and only if $I$ and $J$ satisfy
the descending chain condition (except for trivial cases).

Finally, even if the set of log canonical thresholds with bounded coregularity satisfies the ascending chain condition, 
it is intuitive to study its accumulation points,
meaning, the set of infinite decreasing sequences in the set of log canonical thresholds.
In~\cite{HMX14}, the authors prove that accumulation points of $n$-dimensional log canonical thresholds 
come from $(n-1)$-dimensional log canonical thresholds (up to increasing the coefficients set).
In this direction, we prove that the analogous statement holds for log canonical thresholds with bounded coregularity.

\begin{introthm}\label{introthm:accum-points}
Let $I\subset [0,1]$ be a set satisfying the DCC.
Assume that $1\in I$ is the only accumulation point 
and $I=I^+$.
Then, we have that the accumulation points of 
${\rm LCT}_c(I)$ 
are contained in 
${\rm LCT}_{c-1}(I)$.
\end{introthm}

In the previous statement,
${\rm LCT}(I)$ stands for ${\rm LCT}(I,\zz_{>0})$,
the set $D(I)$ is the derived set of $I$, 
and $I^+$ is the set of positive integral combinations of the elements of $I$ (see Definition~\ref{def:DI}).
In particular, $D(I)$ and $I$ also satisfy the DCC (see Lemma~\ref{lem:DI-DCC}).
The set $D(I)$ emerges from the set $I$ naturally when performing adjunction to a divisor (see Lemma~\ref{lem:coeff-adj}).

We prove many of the theorems in this paper in the context of generalized pairs as in~\cite{BZ16}.
We presented them in the setting of pairs in the introduction to simplify the exposition.

As a counterpart of the coregularity, one can define the {\em regularity} of a klt singularity (see Definition~\ref{def:coreg-general}).
This definition is motivated by Shokurov and Prokhorov's work in the theory of complements for surfaces~\cite{Sho00,Pro01c}.
The regularity of klt singularities has connections with minimal log discrepancies~\cite{Mor20d,Mor21b,Mor21a},
boundedness of klt singularities~\cite{HLM20,Mor18c},
and local fundamental groups~\cite{BFMS20,Mor20c,Mor21}.
Indeed, the regularity bounds the rank of finite abelian groups acting on klt germs.
We expect further developments in the understanding of klt and lc singularities from the perspective of regularity and coregularity. 

\subsection*{On the techniques of the article} 
The idea of using projective geometry to
study log canonical thresholds dates back to McKernan and Prokhorov~\cite{MP03}.
In~\cite{HMX14}, this approach, adjunction theory, and the minimal model program are used to settle
the ascending condition for log canonical thresholds of bounded dimension.
In~\cite{BZ16}, Birkar and Zhang generalized this result to the setting of generalized pairs.
Our article relies on the ideas of both papers~\cite{HMX14,BZ16}.
We aim to reduce Theorem~\ref{introthm:acc-global-coreg}
to the usual global ACC for generalized pairs~\cite[Theorem 1.6]{BZ16}.
To do so, we will apply the strategy of~\cite{HMX14}
inductively on the coregularity.
For this aim, we use the minimal model program and study how the coregularity behaves under adjuntion (Lemma~\ref{lem:coreg-adjunction}) and
by passing to the general fiber of a fibration (Lemma~\ref{lem:coreg-gen-fiber}).
In the following subsection, we expand on the details of the proofs.

\subsection*{Sketch of the proof:}
In this section, we give a sketch of the proof of our main results
Theorem~\ref{introhm:acc-coreg}
and Theorem~\ref{introthm:acc-global-coreg}.

As explained above,
Theorem~\ref{introthm:acc-global-coreg}
implies 
Theorem~\ref{introhm:acc-coreg}
by means of a global-to-local reduction.
However, this requires an argument slightly different than the usual one.
For simplicity,
we consider the local setting
of a log canonical singularity 
$(X,B+t\Gamma;x)$ 
which is strictly log canonical at $x$, 
the coefficients
of $B$ and $\Gamma$ belong to a 
set $I$ satisfying the DCC
and the coregularity of the pair
$(X,B+t\Gamma;x)$ is at most $c$.
Then, we consider a dlt modification
$(Y,B_Y+t\Gamma_Y+\sum_{i=1}^r S_i)$ of the previous pair. 
Here $B_Y$ (resp. $\Gamma_Y$)
is the strict transform of
$B$ (rep. $\Gamma$) on $Y$.
Furthermore, the divisors $S_1,\dots,S_r$ are the divisors extracted by the dlt modification.
We denote by
$\phi\colon Y\rightarrow X$ 
the projective birational morphism.
In order to apply the global result, 
we need to choose a component $S_i$ which intersects $\Gamma_Y$ and perform adjunction to it.
This way, we can use the global ACC 
in order to control the coefficient $t$.
By Lemma~\ref{lem:coreg-adjunction}, we may control the coregularity of the pair induced on $S_i$.
However, to obtain a projective pair, we need to pass to the general fiber 
of $S_i\rightarrow \phi(S_i)$
which may have higher coregularity (see Example~\ref{ex:5}).
In order to fix this issue,
we will consider a minimal log canonical center $Z$ of $(X,B+t\Gamma;x)$ which contains $x$. 
We may choose $Z$ so that the support of $t\Gamma$ contains $Z$.
By possibly blowing up our dlt modification further, we may assume that $S_1$ maps onto $Z$.
By running a suitable MMP on $Y$ over $X$, we may assume that $-S_1$ is nef over the base.
At this step, we may lose the dlt condition of the log pair on $Y$, but it will remain log canonical.
Under the assumption on $S_1$, 
the set-theoretic preimage of $Z$ on $Y$ equals $S_1$.
In particular, $t\Gamma_Y$ must intersect $S_1$
and the intersection
$t\Gamma_Y \cap S_1$ dominates $Z$.
We may take a further dlt modification, keeping the previous conditions (after possibly replacing $S_1$ with some other component of coefficient one).
Thus, we are in the situation that
$t\Gamma_Y \cap S_1$ dominates $Z$.
Since $Z$ is a minimal log canonical center, 
Lemma~\ref{lem:coreg-gen-fiber} 
implies that the coregularity
of the log pair induced on the general fiber of
$S_1\rightarrow Z$ is bounded by $c$. 
Thus, we can use the global ACC to control the coefficient $t$ multiplying
the divisor $\Gamma$.

Now, we proceed to sketch the proof of the global ACC.
Consider a sequence of log Calabi--Yau pairs $(X_i,B_i)$
as in the statement of Theorem~\ref{introthm:acc-global-coreg}.
However, for each one there is a prime component $D_i\subset X_i$, 
so that ${\rm coeff}_{D_i}(B_i)$ violates the ascending chain condition, i.e., it is strictly increasing.
If the sequence of varieties $X_i$ has bounded dimension, 
then the finiteness statement follows from the global ACC proved in~\cite{HMX14}.
The aim is to replace $(X_i,B_i)$ with a dlt modification 
and perform adjunction to a component of
$\lfloor B_i\rfloor$ that intersects $D_i$.
By Lemma~\ref{coregularity-and-dimension},
as far as $\dim X_i>c$, the divisor
$\lfloor B_i\rfloor$ will be non-empty.
A priori, it may happen that 
$\lfloor B_i \rfloor \cap D_i =\emptyset$.
However, we can always run a minimal model program for the pair 
$(X_i,B_i-\epsilon_i S_i)$ for some prime component $S_i\subset \lfloor B_i\rfloor$
and $\epsilon_i$ small enough.
We argue that, after finitely many steps of this MMP, we are in one of the following scenarios:
\begin{itemize}
    \item the strict transform of $D_i$ will intersect the strict transform of $S_i$, or 
    \item we obtain a $\pp^1$-linking structure and both
    the strict transform of $D_i$ and $S_i$ are multi-sections for this bundle.
\end{itemize}
In the former case, we will perform adjunction to $S_i$ and replace $(X_i,B_i)$ with the obtained pair.
In the latter case, we will perform adjunction to a general fiber of the $\pp^1$-linking structure.
When performing adjunction,
we may replace $X_i$ with a variety 
of one dimension lower.
By performing this process for each pair
$(X_i,B_i)$ in the sequence, 
we obtain a new sequence 
$(Z_i,B_{Z_i})$ of bounded dimension
with coefficients on a DCC set $D(I)$.
In Lemma~\ref{lem:CY-coreg-to-dim}, we will prove that the sequence
${\rm coeff}_{D_i}(B_i)$ violating the ACC induces a sequence
of prime divisors $D_{Z_i}\subset Z_i$
so that
${\rm coeff}_{D_{Z_i}}(B_{Z_i})$ also violates
the ascending chain condition.
This leads to a contradiction of the usual global ascending chain condition~\cite{HMX14}.

\subsection*{Acknowledgements}
The authors would like to thank 
Stefano Filipazzi for many useful comments.
This project was initiated in the \href{https://web.math.princeton.edu/~jmoraga/Learning-Seminar-MMP}{Minimal Model Program Learning Seminar}. The first author received partial financial support by the NSF under János Kollár's grant number DMS-1901855.

\section{Preliminaries}

We work over an algebraically closed field $\kk$ of characteristic zero.
Unless otherwise stated, our varieties are normal and quasi-projective over the base field.
In this section, we introduce some preliminary results
regarding generalized pairs, singularities, 
adjunction, coregularity, and the global ascending chain condition.

\subsection{Generalized pairs and singularities}
In this subsection, we recall the basics about birational nef divisors, generalized pairs
and their singularities.

\begin{definition}
{\em 
For a normal quasi-projective variety $X$, consider the set of all birational projective morphisms $\phi:X_\phi \rightarrow X$, with $X_\phi$ normal. A {\em b-divisor} on $X$ is an element of the inverse limit:
\[{\rm Div}_b(X)=\varprojlim_{\phi} {\rm Div}(X_\phi).\]

Therefore a b-divisor $M$ can be written as a countable sum $\sum_i b_iV_i$, where each $V_i$ is a divisorial valuation of $\kk(X)$ such that any normal variety $Y$ birational to $X$ has only finitely many $V_i$'s realized as divisors in $Y$.

The {\em trace} of $M$ on a variety $Y$ birational to $X$ is defined as:
\[M_Y:=\sum_i b_iD_i,\]
where $i$ runs over all valuations $V_i$ with $c_Y(V_i)=D_i$ and ${\rm  codim}_{Y}D_i=1$.

A b-divisor on $X$ is {\em b-$\rr$-Cartier} if there exists a birational model $X'$ of $X$ and a $\rr$-cartier divisor $D_{X'}$, such that for any birational projective morphism $\phi:Y\rightarrow X'$, we have that $M_{Y}=\phi^*D_{X'}$.
The b-divisor 
$M$ is a {\em b-nef} divisor if $D_{X'}$ is nef.
In the previous context, 
we say that $X'$ is a model where $M$ descends
and we write $M'$ for the trace of $M$ on $X'$.
Note that, a model where $M$ descend can always be replaced by higher models by blowing-up.

Let $\psi\colon X'\rightarrow X$ be a model where $M$ descends.
Assume that $M$ is $\rr$-Cartier on $X$.
Assume that $M'$ is a nef $\rr$-Cartier divisor.
By the negativity lemma, we can write $\psi^*M=M'+E$, where $E$ is an effective exceptional divisor.
We say that $X\setminus \psi(E)$ is the locus on $X$ where $M$ {\em descends}
and
we say that $\psi(E)$ is the locus of $X$
where $M$ {\em does not descend}.
}
\end{definition}

\begin{definition}
{\em
A {\em generalized pair} $(X,B+M)$ is a triple, consisting of:

\begin{itemize}
    \item a normal quasi-projective variety $X$;
    \item an effective divisor $B$; and
    \item a b-nef divisor $M$;
\end{itemize}
such that $K_X+B+M$ is $\rr$-Cartier.
We call $B$ and $M$ the boundary and moduli parts, respectively.

For a generalized pair $(X,B+M)$ and a projective morphism $\pi:Y \rightarrow X$, we can write:
\[
\pi ^*(K_X+B+M)=K_Y+B_Y+M_Y.
\]
We will call $(Y,B_Y+M_Y)$ the {\em log pull-back} of $(X,B+M)$ to $Y$.
Note that $B_Y$ is uniquely determined by the previous equality.
}
\end{definition}

\begin{definition}
{\em
Let $(X,B+M)$ be a generalized pair and $\pi:Y \rightarrow X$ a birational projective morphism.
As above, we can write 
\[ K_Y+B_Y+M_Y=\pi ^{*} (K_X +B+M).\]
For a prime divisor $E$ on $Y$, we define the {\em generalized log discrepancy} of $(X,B+M)$ at $E$ to be:

\[a_E(X,B+M)=1-{\rm coeff}_E(B_Y).\]

We say that $(X,B+M)$ has {\em generalized log canonical singularities} (resp. {\em generalized kawamata log terminal singularities}) if $a_E(X,B+M) \geq 0$ (resp. $a_E(X,B+M)>0$) for any prime divisor $E$ over $X$. We write glc (resp. gklt) for short.

}
\end{definition}

\begin{definition}
{\em
A germ $(X;x)$ is said to be of {\em generalized klt type} (resp. {\em generalized lc type}) if there exists a boundary $B$ and a b-nef divisor $M$ on $X$, such that $(X,B+M)$ is generalized klt (resp. generalized lc).
}
\end{definition}

\begin{definition}
{\em
A {\em generalized log canonical place} (resp. {\em generalized non-klt place}) of a generalized pair $(X,B+M)$ is a prime divisor $E$ over $X$, for which $a_E(X,B+M)=0$ (resp. $a_E(X,B+M) \leq 0$).

A {\em generalized log canonical center} (resp. {\em generalized non-klt center}) of a generalized pair $(X,B+M)$ is the center on $X$ of a generalized log canonical place (resp. generalized non-klt place). 
We abbreviate the set of generalized lc centers of $(X,B+M)$ as ${\rm lcc}(X,B+M)$.

A generalized log canonical center that is minimal with respect to inclusion among generalized log canonical centers is said to be a {\em minimal generalized log canonical center}.
}
\end{definition}

\begin{definition}
{\em
A generalized log canonical pair $(X,B+M)$ is {\em generalized divisorially log terminal}, or {\em gdlt} for short, if there exists an open set $U \subset X$, such that:
\begin{itemize}
\item The coefficients of $B$ are less or equal than one;
\item the open set $U$ is smooth and $\lfloor B\rfloor|_U$ has simple normal crossing support;
\item all generalized non-klt centers of $(X,B+M)$ intersect $U$ and consist of strata of $\lfloor B \rfloor$; and 
\item the locus on $X$ where $M$ descends contains the generic points of each strata of $\lfloor B\rfloor$.
\end{itemize}
}
\end{definition}

\begin{definition}
{\em
Let $(X,B+M)$ be a generalized pair and $\phi:Y \rightarrow X$ be a log resolution of $(X,B)$.
Assume that $M$ descends on $Y$.
Let $S$ be the normalization of a component of $\lfloor B \rfloor$ and $S_Y$ its birational transform in $Y$. Let $(Y,B_Y+M_Y)$ be the log pull-back of $(X,B+M)$ and let:
\[K_{S_Y}+B_{S_Y}+M_{S_Y}=(K_Y+B_Y+M_Y)|_{S_Y},\]
with $B_{S_Y}=(B_Y-S_Y)| _{S_Y}$ and $M_{S_Y}=M_{Y}|_{S_Y}$.

Let $f$ be the induced projective birational morphism from $S_Y$ to $S$, $B_S=f_*B_{S_Y}$, and $M_S=f_*M_{S_Y}$.
Then the following equality is called {\em generalized adjunction}:
\[K_S+B_S+M_S=(K_{X}+B+M)|_S.\]
\em}
\end{definition}

\begin{definition}
{\em
Let $(X,B+M)$ be a generalized log canonical pair,
let $\Gamma\geq 0$ be an effective divisor,
and let $N$ be a b-nef divisor, such that 
$\Gamma+N$ is $\rr$-Cartier.
Then, the {\em generalized log canonical threshold}
of $\Gamma+N$ 
with respect to $(X,B+M)$ is:
\[ 
{\rm lct}(X,B+M; \Gamma+ N)=\sup \{t \in \rr \mid (X,B+M+t(\Gamma+N)) \text{ is generalized log canonical}\}. 
\]
In the previous generalized pair,
$B+t\Gamma$ is the boundary divisor, 
and $M+tN$ is the moduli divisor.
We define ${\rm LCT}_c(I,J)$ to be the set of all 
${\rm lct}(X,B+M;\Gamma+N)$, such that the following conditions are satisfied:
\begin{enumerate}
    \item the coefficients of $B$ are in $I$;
    \item the coefficients of $\Gamma$ are in $J$;
    \item we can write $M'=\sum_j\lambda_j M'_j$
    where $\lambda_j\in I$ and each $M'_j$ is nef Cartier;
    \item we can write 
    $N'=\sum_j\mu_j N'_j$ where $\mu_j\in J$ and each $N'_j$ is nef Cartier; and
    \item the coregularity of 
    $(X,B+M+t(\Gamma+N))$ is at most $c$.
\end{enumerate}
The previous set is called the set of log canonical thresholds with indices in $I$ and $J$ and coregularity bounded above by $c$.
\em}
\end{definition}

\subsection{Generalized dlt modifications}
In this subsection, we recall the existence of generalized dlt modifications
and results regarding the coefficients obtained after adjunction.
Some of the results in this subsection are analogous to those in~\cite{HMX14,BZ16}.
In some cases, we give a short proof for the sake of completeness.

\begin{definition}
{\em
Let $(X,B+M)$ be a generalized pair, then a generalized pair $(Y,B_Y+M_Y)$ is a {\em generalized dlt modification} of $(X,B+M)$ if $(Y,B_Y+M_Y)$ is a  $\qq$-factorial gdlt pair and there exists a projective birational morphism $\phi:Y \rightarrow X$ satisfying the following conditions:
\begin{enumerate}
    \item the exceptional locus 
    ${\rm Ex}(\phi)$ is purely divisorial;
    \item all the divisors on $Y$ exceptional over $X$ satisfy
    $a_E(X,B+M)=0$; and 
    \item we have that 
    $K_Y+B_Y+M_Y=\phi^*(K_X+B+M)$.
\end{enumerate}
}
\end{definition}

The following lemma is well-known to the experts (see, e.g.~\cite[Theorem 2.9]{FS20}).

\begin{lemma}
Let $(X,B+M)$ be a generalized log canonical pair.
Then there exists a generalized dlt modification
$(Y,B_Y+M_Y)$ of $(X,B+M)$.
\end{lemma}

\begin{definition}\label{def:DI}
{\em
For a set $I \subset \rr$, we define:
\[I^+=\{0\} \cup 
\left\{i \in [0,1] \mid i=\sum_{k=1}^{\ell} j_k, \text{ for some $j_1,\ldots,j_{\ell} \in I$} \right\},\]
and 
\[D(I)=\left\{a \leq 1 \mid a=\frac{m-1+f}{m}, m\in \zz_{>0}, f \in I^+\right\}.\]
}
\end{definition}

The following lemmas are straightforward from the previous definitions, see for example \cite[4.4]{MP04}

\begin{lemma}\label{lem:DI-DCC}
Let $I$ be a set of nonnegative real numbers satisfying the DCC, then the sets  $D(I)$ and $I^+$
satisfy the DCC.
\end{lemma}

\begin{lemma}
Let $I$ be a set of nonnegative real numbers, then we have that:
\[D(D(I))=D(I)\cup \{1\}.\]
\end{lemma}

\begin{definition}\label{def:D_dI}
{\em
Let $I$ be a set of nonnegative real numbers.
Let $d$ be a real number in the interval $[0,1]$.
We define:
\[
D_d(I) = \left\{a \leq 1 \mid a=\frac{m-1+f+kd}{m}, k,m \in \zz_{>0}, f\in I^+ \right\} .
\]
}
\end{definition}

\begin{lemma}\label{lem:D_d-induction}
Let $I$ be a set of nonnegative real numbers.
Let $d_1 \in D_d(I)$, 
then we have that
$D_{d_1}(I)\subseteq D_d(I)$.
\end{lemma}

\begin{proof}
Any element in $D_{d_1}(I)$ is of the form:
\begin{align*}
a &= \frac{m-1+f+k_1d_1}{m},
\end{align*}
where $m,k_1\in \zz_{>0}$
and 
$f\in I^+$.
Then, we can write
\begin{align*} 
a&=\frac{m-1+f+k_1\left(\frac{n-1+g+kd}{n}\right)}{m}\\ &=\frac {mn-n+fn+k_1n-k_1+k_1g+k_1kd}{mn} \\ &=\frac{mn-1+(k_1-1)(n-1)+fn+k_1g+k_1kd}{mn},  
\end{align*}
which is in $D_d(I)$.
\end{proof}

The following lemma is similar to~\cite[Lemma 5.2]{HMX14}.

\begin{lemma}\label{lem:finite-to-finite}
Let $J$ be a finite set and $I$ a set of nonnegative real numbers satisfying the DCC.
Then, the set 
\[
I_1 = 
\left\{
i\in I  \mid \frac{m-1+ki+f}{m}\in J \cap [0,1] \text{ for some }m,k\in \zz_{>0}, f\in I^+
\right\}
\] 
is also finite.
\end{lemma}

\begin{proof}
We can assume $i\neq 0$.
We have that $ki \leq 1$, so no element of $i\in I_1$ can be greater than one, hence we can assume $I\subset [0,1]$.
As $I$ satisfies the DCC, we have that $i$ is bounded below, therefore as $ik\leq 1$, we have that $k$ is bounded above, so it can only take  finitely many values.

Call $\ell=\frac{m-1+ki+f}{m}.$ If $\ell=1$, then we can assume $m=1$, otherwise we have that $\ell$ is bounded above by a number smaller than one, since $J$ is finite. Therefore $1-\frac{1}{m}$ is bounded above, and $m$ can only take finitely many values.

For any fixed $m,\ell,k$, the set:
\[
I_{m,\ell,k}:=\left \{i\mid i=\frac{m\ell-m+1-f}{k}, f\in I^+\right\}
\]
is a subset of $I$. Hence $I_{m,\ell,k}$ satisfies the DCC. The set $I_{m,\ell,k}$ also satisfies the ACC as the only variable $f\in I$ satisfies the DCC. Therefore, $I_{m,\ell,k}$ is a finite set. As $m,\ell,k$ can only take finitely many possible values, the set $I_1$ is also finite.   
\end{proof}

The following lemma follows from the proof of \cite[Proposition 4.9]{BZ16}.
It allows us to control the coefficients of the boundary part
and the moduli part under adjunction to a divisorial log canonical center.

\begin{lemma}\label{lem:coeff-adj}
Let $I$ be a set of nonnegative real numbers satisfying the DCC.
Let $(X,B+M)$ be a generalized dlt pair, with coefficients in $I$. 
Let $d$ be a positive real number for which either:
\begin{itemize}
    \item there exists a prime component $D$ on $X$ with
    ${\rm coeff}_D(B)=d$, or
    \item we can write $M'=\sum_j \lambda_j M'_j$ with $\lambda_j$ nonnegative real numbers, each $M'_j\not\equiv 0$ is nef Cartier, and
    $\lambda_j=d$ for some $j$.
\end{itemize}
Let $S$ be the normalization
of a component of $\lfloor B\rfloor$
with either $S\cap D\neq \emptyset$ or
$S\cap M_j\neq \emptyset$ with $\lambda_j=d$.
Let $B_S$ and $M_S$
be the boundary and moduli part
defined by the generalized adjunction.
Then, we have that either:
\begin{itemize}
    \item[(i)] the coefficients of $B_S$ belong to $D(I)$,
    \item[(ii)] for any prime component $Q$ of $D\cap S$ we have that ${\rm coeff}_Q(B_S)\in D_d(I)$, or 
    \item[(iii)] we can write 
    $M'_S=\sum_j \eta_j M'_{S,j}$ with $\eta_j$ nonnegative real numbers, each $M'_{S,j}\not\equiv 0$ is nef Cartier, and 
    $\eta_j=d$ for some $j$.
\end{itemize}
\end{lemma}

\begin{proof}
Let $(S,B_S+M_S)$ be the pair
obtained by generalized adjunction of $(X,B+M)$ to $S$ (see, e.g.,~\cite[Definition 4.7]{BZ16}).
Condition (iii) follows from the definition of generalized adjunction.
Hence, we need to prove the first and second statement.
Since this statement only depends on codimension one points of $S$, 
we may assume that $X$ has dimension $2$
and it is $\qq$-factorial.
By~\cite[Corollary 3.10]{Sho92}, we can write 
\[
K_S+B'_S = (K_X+B)|_S,
\] 
where the coefficients of 
$B'_S$ belong to $D(I)$.
More precisely, if $Q$ is a prime divisor of $S$, we can write
\begin{equation}
\label{eq:coeff-under-adj}
{\rm coeff}_P(B'_S)=1-\frac{1}{m}+\frac{\sum_i m_i\lambda_i}{m},
\end{equation} 
where the $m$ is the orbifold index of $X$ at $Q$, 
the $\lambda_i$'s belong to $I$,
and $m_i$'s are positive integers.
In the previous formula, 
$\lambda_i$ is the coefficient in $B$ of a prime divisor $P$ on $X$ 
amd $m_i$ is the multiplicity of $P$ at $Q$.
Let $M_j$ be the push-forward of $M'_j$ to $X$.
Then, we have that
$mM=\sum_j \lambda_j mM'_j$.
Each $mM'_j$ is Cartier on a neighborhood of $Q$.
Let $\pi\colon X'\rightarrow X$ be the projective birational morphism.
By the negativity lemma,
we can write 
$\pi^*(mM) = \sum_j \lambda_jmM_j + \lambda_j E_j$ where each $E_j$ is Cartier.
Hence, we have that 
\[
{\rm coeff}_P(B_S)= 
1-\frac{1}{m}+\frac{\sum_i m_i\lambda_i + \sum_j \lambda_j{\rm mult}_Q(E_j)}{m}.
\] 
This finishes the first statement.

Finally, if there is a prime component $D$ of $X$
for which ${\rm coeff}_D(B)=d$, then for any prime $Q$ of $D\cap S$ we have ${\rm coef}_Q(B'_S)\in D_d(I)$ by~\eqref{eq:coeff-under-adj}.
Thus, we have that
${\rm coeff}_Q(B_S)\in D_d(I)$ as well.
This finishes the proof of the second statement.
\end{proof} 

The following lemma allows to keep track of a prime divisor $D$ when passing to a dlt modification.

\begin{lemma}\label{lem:dlt-intersection}
Let $(X,B+M)$ be a generalized log canonical pair.
Let $D$ be a prime divisor on $X$ which satisfies:
\[
{\rm Supp}(D)\cap {\rm lcc}(X,B+M) \neq \emptyset.
\]
Let $(Y,B_Y+M_Y)$ be a generalized dlt modification of $(X,B+M)$.
Let $D_Y$ be the strict transform of $D$.
Then, we have that 
\[
D_Y \cap \lfloor B_Y \rfloor \neq \emptyset.
\]
In particular, there exists 
a prime component
$S \subset \lfloor B_Y \rfloor$  with $D_Y \cap S \neq \emptyset$.
\end{lemma}

\begin{proof}
Let $\phi:Y \rightarrow X$ be the dlt modification. 
Let $p\in D\cap {\rm lcc}(X,B+M)$. By definition we have that $D_Y\cap \phi^{-1}(p)\neq \emptyset$. 
If ${\rm dim } (\phi^{-1}(p)) \geq 1$, we have that $\phi^{-1}(p)$ is contained in an exceptional divisor, therefore $\phi^{-1}(p) \subset \lfloor B_Y\rfloor$, as $\phi$ is a dlt modification. Hence $\emptyset \neq D_Y\cap \phi^{-1}(p) \subset D_Y \cap \lfloor B_Y \rfloor$.
Otherwise $\phi^{-1}(p)$ is a point, then $\phi^{-1}(p)\in {\rm lcc}(Y,B_Y+M_Y)=\lfloor B_Y\rfloor$, hence $\phi^{-1}(p)\in D_Y \cap \lfloor B_Y \rfloor$.
\end{proof}

\subsection{Regularity and coregularity}
In this subsection, we introduce the regularity
and coregularity.
We prove some lemmas regarding the coregularity.

\begin{definition}\label{def:minimal-dlt-center}
{\em 
Let $(X,B+M)$ be a generalized log canonical pair.
A {\em minimal dlt center} of $(X,B+M)$
is a minimal log canonical center of any generalized
dlt modification of $(X,B+M)$.
}
\end{definition}

\begin{definition}\label{def:dual-complex}
{\em 
Let $(X,B+M)$ be a generalized pair.
Let $\pi: Y \to X$ be a generalized dlt modification such that
\[
K_Y + B_Y + M_Y = \pi^*(K_X+B+M),
\]
where $M_Y$ is the pushforward of $M'$ on $Y$. Write \[\lfloor B_Y\rfloor = E_1 + E_2 + \cdots + E_r,\] which is a simple normal crossing divisor on $Y$. 

The {\it dual complex} $\mathcal{D}(Y,B_Y+M_Y)$
of the dlt modification is a simplicial complex constructed as follows:
\begin{itemize}
    \item Every divisor $E_i$ corresponds to a vertex $v_i$ in $\mathcal{D}(Y,B_Y+M_Y)$. For every subset $I\subseteq \{1,2,\ldots,r\}$ and every irreducible component $Z$ of $\bigcap_{i\in I} E_i$, $Z$ corresponds to a simplex $v_Z$ of dimension $\# I - 1$.
    \item For every subset $I\subseteq \{1,2,\ldots,r\}$ and every $j\in I$, we have the following gluing maps. Let $Z\subseteq \bigcap_{i\in I} E_i$ be any irreducible component, then $Z$ lies in a unique component $W$ of $\bigcap_{i\in I\setminus\{j\}}E_i$. We have a gluing map which is the inclusion of $v_W$ into $v_Z$, as the face of $v_Z$ which do not contain vertex $v_i$.
\end{itemize}
The dimension of $\mathcal{D}(Y,B_Y+M_Y)$ is defined to be the smallest dimension
of the maximal simplex, with respect to the inclusion,
of $\mathcal{D}(Y,B_Y+M_Y)$.

We define the \textit{dual complex} $\mathcal{D}(X,B+M)$ associated to $(Y,B_Y+M_Y)$ as $\mathcal{D}(Y,B_Y+M_Y)$ and we define
\[
\dim \mathcal{D}(X,B+M) = \dim \mathcal{D}(Y,B_Y+M_Y).
\]
If $Z\subset X$ is an irreducible subvariety with generic point $\eta_Z$, 
we define $\mathcal{D}(X,B+M;\eta_Z)$ to be $\mathcal{D}(U, B|_U+M|_U)$ for a sufficiently small neighborhood $U$ of $\eta_Z$.
}
\end{definition}

\begin{proposition}
Let $(X,B+M)$ be a generalized log canonical pair. Then $\dim \mathcal{D}(X,B+M)$ does not depend on the choice of generalized dlt modifications $(Y,B_Y+M_Y)$ of $(X,B+M)$.
\end{proposition}

\begin{proof}
Let $\pi_i: (Y_i,B_{Y_i}+M_{Y_i})\to (X,B+M)$ be generalized dlt modifications, for $i= 1,2$. Suppose $\dim \mathcal{D}(Y_1,B_{Y_1}+M_{Y_1}) = r$. Then there exists $(r+1)$ $\pi_1$-exceptional divisors whose intersection is non-empty.
Let $W$ be an irreducible component of this intersection and $Z = \pi_1(W)$. 
Then $Z$ is a minimal generalized log canonical center of $(X,B+M)$. 
So $\dim \mathcal{D}(X,B+M;\eta_Z) = r$. 
By~\cite[Theorem 1.1 and Theorem 1.6]{FS20}, the dimension of $\mathcal{D}(X,B+M;\eta_Z)$ is independent of the choice of the dlt modification.
Hence, $\dim \mathcal{D}(Y_2, B_{Y_2}+M_{Y_2})\leq r$. Similarly, we can prove that $\dim \mathcal{D}(Y_2, B_{Y_2}+M_{Y_2}) \geq \dim \mathcal{D}(Y_1,B_{Y_1}+M_{Y_1})$ and hence the equality holds.
\end{proof}

\begin{definition}\label{def:coreg-general}
{\em 
Let $(X,B+M)$ be a quasi-projective normal generalized pair.
We define the {\em regularity} of $(X,B+M)$ to be 
\[
{\rm reg}(X,B+M) := \dim\mathcal{D}(X,B+M).
\]
On the other hand, we define the {\em coregularity}
of $(X,B+M)$ to be
\[
{\rm coreg}(X,B+M):= \dim X -\dim \mathcal{D}(X,B+M)-1.
\]
Hence, by definition, we have that 
\[
{\rm reg}(X,B+M) + {\rm coreg}(X,B+M)= \dim X-1.
\] 
In a few words, a generalized log canonical pair $(X,B+M)$ has coregularity $c$ 
if for every minimal log canonical center of $(X,B+M)$,
we may find a minimal dlt center of dimension $c$ mapping onto it.

Analogously, for $Z\subset X$ an irreducible subvariety, we define
\[
\reg(X,B+M; \eta_Z) = \reg (U, B|_U+M|_U)
\]
for a sufficiently small neighborhood $U$ of $\eta_Z$.
}
\end{definition}

Now, we turn to define the absolute regularity
and absolute coregularity
of klt singularities and Fano type varieties. 
This definition is not crucial for this paper, 
as in the log canonical and log Calabi--Yau case
they agree with the usual regularity
and coregularity (see Remark~\ref{rem:abs-vs-coreg}).
However, it is worth mentioning as 
it will be considered in forthcoming research.

\begin{definition}
{\em 
Let $(X,B+M)$ be a projective generalized pair.
We say that $(X,B+M)$ is {\em generalized log Calabi--Yau}
if the following conditions are satisfied: 
\begin{itemize}
    \item $(X,B+M)$ has generalized log canonical singularities, and
    \item $K_X+B+M\sim_\qq 0$.
\end{itemize}
We say that a generalized pair $(X,B+M)$ 
is of {\em log Calabi--Yau type} if there exists $\Gamma\geq0$
for which $(X,B+\Gamma+M)$ is a generalized log Calabi--Yau pair.
In this case, to abbreviate, 
we may say that $(X,B+M)$ is CY type.
For a generalized log Calabi--Yau projective variety
$(X,B+M)$, 
two minimal dlt center are birationally equivalent.
}
\end{definition}

\begin{definition}
\label{def:abs-reg}
{\em 
Let $(X,B+M)$ be a projective generalized pair.
We define the {\em absolute regularity} of $(X,B+M)$ to be
\[
\hat{{\rm reg}}(X,B+M):=
\max \{
\dim \mathcal{D}(X,B+\Gamma+M) \mid 
(X,B+\Gamma+M) \text{ is generalized log Calabi--Yau}
\}.
\]
If the dual complex 
$\mathcal{D}(X,B+\Gamma+M)$ is empty, then
we set 
$\dim\mathcal{D}(X,B+\Gamma+M)=-1$.
If the previous set is empty, then we set 
${\rm reg}(X,B+M)=-\infty$.
Hence, a generalized pair has finite absolute regularity
if and only if it is of generalized log Calabi--Yau type.
A generalized pair has absolute regularity $-1$
if and only if every log Calabi--Yau structure on it has klt singularities, i.e., 
absolute regularity $-1$ is equivalent to being exceptional.
The absolute regularity is non-negative if and only if we can find a log Calabi--Yau structure with a log canonical center.
If $X$ is $n$-dimensional, we have that 
\[
\hat{{\rm reg}}(X,B+M) \in \{-\infty,-1,0, \dots, n-1\}.
\]
We define the {\em absolute coregularity} of $(X,B+M)$ to be
\[
\hat{{\rm coreg}}(X,B+M):= \dim X - \hat{{\rm reg}}(X,B+M)-1.
\]
Analogously, a generalized pair
has finite absolute coregularity
if and only if it is of generalized log Calabi--Yau type.
A generalized pair $(X,B+M)$ has absolute coregularity $\dim X$
if and only if every log Calabi--Yau $(X,B+\Gamma+M)$ structure has klt singularities, i.e., 
the generalized pair $(X,B+M)$ is exceptional.
The absolute coregularity $c$ is in the interval 
$\{0,\dots,n-1\}$ if and only if
there exists a log Calabi--Yau structure
$(X,B+\Gamma+M)$ which admits a log canonical center.
In summary, for
a projective variety $X$ of dimension $n$
and a generalized pair structure $(X,B+M)$, we have that 
\[
\hat{{\rm coreg}}(X,B+M)\in \{0,1,2,\dots,n,\infty\}.
\]
If the absolute coregularity of $(X,B+M)$ is in the set
$\{0,1,2,\dots,n\}$ then the absolute coregularity equals the dimension of the smallest minimal dlt center
on a log Calabi--Yau structure of $(X,B+M)$.
If the log Calabi--Yau structure $(X,B+M)$
is itself klt, then we can set $X$ formally as a log canonical center.
}
\end{definition}

\begin{definition}
{\em 
Let $(X,B+M;x)$ be a generalized log canonical singularity. 
We define the {\em absolute regularity} of $(X,B+M)$ at $x$ to be 
\[
\hat{{\rm reg}}(X,B+M;x):=
\max 
\{
\dim\mathcal{D}(X,B+\Gamma+M;x) \mid 
(X,B+\Gamma+M;x) \text{ is generalized log canonical}
\}.
\]
We define the absolute coregularity of the klt to be
\[
\hat{{\rm coreg}}(X,B+M;x):= \dim X-\hat{\rm reg}(X,B+M;x)-1.
\]
}
\end{definition}

The following remark says that in our setting, 
the coregularity
and absolute coregularity agree.

\begin{remark}\label{rem:abs-vs-coreg}
{\em 
If $(X,B+M)$ is a generalized log Calabi--Yau pair, then 
\[
{\rm reg}(X,B+M)=\hat{\rm reg}(X,B+M) \text{ and }
{\rm coreg}(X,B+M)=\hat{\rm coreg}(X,B+M).
\]
Analogously, if $(X,B+M;x)$ is a generalized log
canonical singularity for which $x$ is a log canonical center, then
\[
{\rm reg}(X,B+M;x)=\hat{\rm reg}(X,B+M;x) \text{ and }
{\rm coreg}(X,B+M;x)=\hat{\rm coreg}(X,B+M;x).
\] 
}
\end{remark}

The following lemma, 
which is straightforward from the definition,
allows us to perform adjunction
whenever the dimension is larger than the coregularity.

\begin{lemma}\label{coregularity-and-dimension}
Let $(X,B+M)$ be a 
generalized log canonical pair.
If $\dim X>{\rm coreg}(X,B+M)$, then 
$(X,B+M)$ has a non-trivial generalized log canonical center.
In particular, if $(Y,B_Y+M_Y)$ is a generalized dlt modification of $(X,B+M)$, then $\lfloor B_Y\rfloor \neq \emptyset$.
\end{lemma}

Now, we turn to prove some preliminaries regarding how the coregularity behaves
under adjunction,
perturbing the boundary, 
and taking fibrations.
The proofs are well-known to the experts. In some cases, we give a short argument.
The following lemma follows from~\cite[Theorem 1.6]{FS20}.

\begin{lemma}\label{lem:coreg-adjunction}
Let $(X,B+M)$ be a generalized log Calabi--Yau pair.
Let $S$ be the normalization of a component
of $\lfloor B\rfloor$.
Let $B_S$ 
and $M_S$ be the boundary and moduli part defined by generalized adjunction, so that
\[
(K_X+B+M)|_S=K_S+B_S+M_S.
\] 
Then, we have that 
${\rm coreg}(S,B_S+M_S)={\rm coreg}(X,B+M)$.
\end{lemma}

\begin{lemma}\label{lem:coreg-throw-components}
Let $(X,B+M)$ be a generalized log canonical pair
and $V$ be a minimal generalized log canonical center of $(X,B+M)$. 
Let $B_0 \leq B$ be the sum of all the components of $B$ containing $V$.
Write $M'=\sum_j M'_j$ where each $M'_j$ is a nef Cartier divisor.
Let $M'_0$ be the sum of the $M'_j$
that do not descend over the generic point of $V$
and let $M_0$ the push-forward of $M'_0$ to $X$.
Then, we have that 
\[
{\rm coreg}(X,B+M;\eta_V)=
{\rm coreg}(X,B_0+M_0;\eta_V).
\] 
\end{lemma}

\begin{proof}
Since the statement is about coregularity at the generic point of $V$, 
we may just localize at the generic point of $V$
and assume that $B_0=B$ and $M_0=M$.
\end{proof}
    
\begin{lemma}\label{lem:coreg-gen-fiber}
Let $(X,B+M)$ be a generalized log canonical pair
and $X\rightarrow Z$ be a fibration.
Assume that every generalized non-klt center of $(X,B+M)$ dominates $Z$.
Let $F$ be a general fiber of $X\rightarrow Z$
and $(F,B_F+M_F)$ be the generalized pair
induced by adjunction.
Then, we have that
\[
{\rm coreg}(F,B_F+M_F)
\leq
{\rm coreg}(X,B+M)-\dim Z.
\]
\end{lemma}

\begin{proof}
We may assume that $(X,B+M)$ is a generalized dlt pair.
Let $d$ be the dimension of $X$.
Since $(X,B+M)$ has coregularity $c$, we can find 
$d-c$ components $S_1,\dots,S_{d-c}$ of $\lfloor B\rfloor$
so that 
\[
S_1\cap \dots \cap S_{d-c}\neq \emptyset.
\] 
Note that $S_{F_i}:=S_i\cap F$
is a component
with coefficient one in $B_F$.
Then, we conclude that
\[
S_{F,1}\cap \dots \cap S_{F,d-c} \neq \emptyset,
\]
as divisors on $F$.
Thus, we have that 
\[
{\rm coreg}(F,B_F+M_F)\leq {\rm coreg}(X,B+M) - \dim Z
\]
as claimed.
\end{proof}

The following lemma states that the coregularity of a generalized log Calabi--Yau pair does not change when we run a 
minimal model program.

\begin{lemma}\label{lem:coreg-mmp}
Let $(X,B+M)$ be a generalized log Calabi--Yau pair. 
Let $X\dashrightarrow X'$ be a birational contraction.
Denote by $B'$ the push-forward of $B$ on $X'$
and by $M'$ the trace on $X'$ 
of the b-nef divisor associated to $M$.
Then, we have that 
\[
{\rm coreg}(X,B+M)=
{\rm coreg}(X',B'+M').
\]
\end{lemma}

\begin{proof}
Since $(X,B+M)$ is generalized log Calabi--Yau, we can find a common log resolution $p\colon Y\rightarrow X$
and $q\colon Y\rightarrow X'$ for which 
\[
K_Y+B_Y+M_Y=p^*(K_X+B+M)=q^*(K_{X'}+B'+M').
\]
By definition, we have that 
$\mathcal{D}(X,B+M)=\mathcal{D}(Y,B_Y^{=1})$,
where 
$B_Y^{=1}$ denotes the sum of the components of $B_Y$ which appear with coefficient one in $B_Y$. 
The same applies to $(X',B'+M')$.
Then, we conclude that 
$\dim \mathcal{D}(X',B'+M')=\dim \mathcal{D}(X,B+M)$.
This concludes the proof.
\end{proof}

\subsection{Global ascending chain condition}

In this subsection, we recall the ascending chain condition
for generalized pairs. 
We will use the global ascending chain condition for
varieties of fixed dimension as
the base for an inductive process in the following section.
The following result is~\cite[Theorem 1.5]{BZ16}.

\begin{proposition}\label{prop:global-acc-dim-n}
Let $n$ be a positive integer.
Let $I$ be a set of nonnegative real numbers satisfying the descending chain condition.
Then there is a finite subset $I_0 \subseteq I$ such that if $(X,B+M)$ satisfies the following conditions:
\begin{enumerate}
    \item the variety $X$ is a projective variety of dimension $n$;
    \item the generalized pair $(X,B+M)$ is generalized log canonical;
    \item the coefficients of $B$ belong to $I$;
    \item we can write $M'=\sum_j \mu_j M'_j$ where
    $\mu_j\in I$ and each $M'_j$ is a nef Cartier divisor;
    \item $\mu_j=0$ whenever $M'_j\equiv 0$; and 
    \item $K_X+B+M\equiv 0$;
\end{enumerate}
then the coefficients of
$B$ and the $\mu_j$ belong to $I_0$.
\end{proposition}

We note that condition $(4)$ is written to avoid trivial counter-examples where we add trivial summands to $M'$ with
arbitrary coefficients in $I$.
Throughout the rest of the paper, we may assume that no
$M'_j$ is numerically trivial,
unless otherwise stated.

\section{Global ascending chain condition}

In this section, we prove the main projective result of this paper. 
We prove the global ascending chain condition for 
log Calabi--Yau pairs
of bounded coregularity.

\begin{theorem}\label{thm:global-acc-coreg-c}
Fix a positive integer $c$ and a set $I$ of nonnegative real numbers satisfying the DCC. Then there exists a finite subset $I_0\subseteq I$ such that if $(X,B+M)$ satisfies the following conditions:
\begin{enumerate}
    \item the generalized pair $(X,B+M)$ is generalized log canonical;
    \item the coefficients of $B$ belong to $I$;
    \item we can write $M'=\sum_j \mu_j M'_j$ where $\mu_j\in I$
    and each $M'_j$ is a nef Cartier divisor; 
    \item the numerical equivalence $K_X+B+M\equiv 0$ holds; and
    \item the variety $X$ is projective with an upper bound ${\rm coreg}(X,B+M)\leq c$;
\end{enumerate}
then ${\rm coeff}(\Delta)\subseteq I_0$
and the $\mu_j$ belong to $I_0$.
\end{theorem}

The following lemma allows to replace a log Calabi--Yau pair $(X,B+M)$ by one whose dimension is at most the coregularity of $(X,B+M)$. 
Furthermore, we are able to control the coefficients of the boundary and moduli part of the new log Calabi--Yau pair.

\begin{lemma}\label{lem:CY-coreg-to-dim}
Given a generalized log Calabi--Yau pair $(X,B+M)$ satisfying the following conditions:
\begin{enumerate}
    \item The coefficients of $B$ belong to $I$;
    \item we can write $M'=\sum_j \mu_j M'_j$ where $\mu_j\in I$
    and each $M'_j$ is a nef Cartier divisor; and 
    \item the variety $X$ is projective with an upper bound ${\rm coreg}(X,B+M)\leq c$.
\end{enumerate}
Let $d$ be either one of the $\mu_j$
or in ${\rm coeff}(B)$.
Then, we can construct a generalized log Calabi--Yau pair 
$(Z,B_Z+M_Z)$ of coregularity at most $c$
satisfying the following conditions:
\begin{enumerate}
    \item[(i)] the variety $Z$ has dimension at most $c+1$;
    \item[(ii)] the coefficients of $B_Z$ are in $D(I)$; and
    \item[(iii)] we can write 
    $M'_Z=\sum_j \eta_j M'_{Z,j}$ where
    $\eta_j\in I$ and each $M'_{Z,j}$ is a nef Cartier divisor.
\end{enumerate}
Furthermore, 
we have that either some
$d'\in {\rm coeff}(B_Z)$
satisfies $d'\in D_d(I)$
with $d'<1$ 
or $\eta_j=d$ for some $j$.
If $d\in {\rm coeff}(B)$, then the former case holds.
\end{lemma}

\begin{proof}
By Lemma~\ref{coregularity-and-dimension}, if
$(X,B+M)$ is generalized klt, 
then the dimension of $X$ is at most $c$, so we are done.
Therefore, we can assume that $(X,B+M)$ is generalized lc but not generalized klt. By Lemma~\ref{lem:D_d-induction}, we only have to prove that there exists a pair of dimension strictly smaller than $\dim X$
satisfying the conditions (i)-(iii).
Inductively, by Lemma~\ref{lem:D_d-induction} we get the desired pair.

We start by taking a generalized dlt modification $(Y,B_Y+M_Y)$ of $(X,B+M)$.
Let $S_Y$ be a prime component of $\lfloor B_Y\rfloor$.
Without loss of generality, we assume that 
$d \in {\rm coeff}(B_Y)$. 
The case in which $d$ is one of the $\mu_j$
proceeds analogously.
We call $D_Y$ a prime component of 
the divisor $B_Y$ that has coefficient $d$.
Pick $\epsilon>0$ to be small enough, we run a $(K_Y+B_Y-\epsilon S_Y+M_Y)$-MMP.
As $K_Y+B_Y+M_Y \equiv 0$, all the steps of the minimal model program are $S_Y$-positive.
By Lemma~\ref{lem:coreg-mmp} the coregularity of $(Y,B_Y+M_Y)$ remains constant.\\

\noindent \textit{Claim:} The divisor $S_Y$ does not get contracted by any step of the minimal model program.

\begin{proof} 
Suppose $S_Y$ gets contracted.
Without loss of generality,
as all our conditions are preserved by the MMP, 
we may assume it is contracted by the very first step of the MMP.
Then, we have a birational contraction
$\phi \colon Y\rightarrow Y'$.
As $K_Y+B_Y+M_Y \equiv 0$, we have that 
\[K_Y+B_Y+M_Y =\phi^{\ast} \phi_{\ast} (K_{Y}+B_{Y}+M_{Y})= \phi ^{\ast} (K_{Y'}+B_{Y'}+M_{Y'}). \]
Here, as usual, $B_{Y'}$ (resp. $M_{Y'}$)
is the push-forward of $B_Y$ (resp $M_Y$) to $Y'$.
Therefore $a_{S_Y}(Y',B_{Y'}+M_{Y'})=0$. However, we have 
\[
\epsilon =a_{S_Y}(Y,B_Y -\epsilon S_Y+M_Y)\leq a_{S_Y}(Y',B_{Y'}+M_{Y'}),
\]
as the log discrepancies only increase after a step of the MMP. A contradiction.
\end{proof}

As $K_Y+B_Y+M_Y \equiv 0$ and $Y$ is projective, we have that $K_Y+B_Y-\epsilon S_Y +M_Y$ is not pseudoeffective.
By~\cite[Lemma 4.4]{BZ16}, we know that this minimal model program
must terminate with a Mori fiber space.
We will split into different cases,
depending on the steps and outcome of the minimal model program.\\

\noindent\textit{Case 1:} The strict transform of the divisor $D_Y$ gets contracted by a step of the MMP.\\

Let $\pi_1:Y_1 \rightarrow Y_2$ be the step where the strict transform of $D_Y$ is contracted.
We denote by $B_1$ (resp. $M_1$) the push-forward
of $B_Y$ (resp. $M_Y$) to the model $Y_1$.
Analogously, we denote by $S_1$ (resp. $D_1$) the push-forward of $S_Y$ (resp. $D_Y$) on $Y_1$.
As $\pi_1$ contracts $(K_{Y_1}+B_1-\epsilon{S_1}+M_1)$-negative curves, we have that $\pi_1$ contracts $S_1$-positive curves. As some contracted curve $C$ must be contained in $D_1$, we have that 
\[
D_1\cap S_1 \supseteq C \cap S_1 \neq \emptyset.
\]
By Lemma~\ref{lem:dlt-intersection},
may replace $(Y_1,B_1+M_1)$ with a dlt modification
and assume that $D_1$ intersects a component $S$
of $\lfloor B_1\rfloor$. Let $(S,B_{S}+M_{S})$ be the generalized pair obtained by adjunction of $K_{Y_1}+B_1+M_1$ to $S$.
By Lemma~\ref{lem:coreg-adjunction}, $(S,B_{S}+M_{S})$ has coregularity at most $c$.

The following conditions are satisfied:
\begin{enumerate}
    \item[(i)] the generalized pair
    $(S,B_{S}+M_{S})$ is generalized log Calabi--Yau, 
    \item[(ii)] the coefficients of $B_{S}$ belong to $D(I)$, and
    \item[(iii)] we can write
    $M_{S'}=\sum_j \eta_j M_{S',j}$ 
    where $\eta_j\in I$
    and each $M_{S',j}$ is a nef Cartier divisor.
\end{enumerate}
Observe that  $D_1\cap S\neq\emptyset$
and ${\rm coeff}_{D_1}(B_1)=d$.
By Lemma~\ref{lem:coeff-adj},
we conclude that any prime component $Q$
of $D_1\cap S$ will satisfy 
that
$d'={\rm coeff}_{Q}(B_S)\in D_d(I)$.
Since we are performing generalized dlt adjunction,
we have that $d'<1$.
This finishes the proof in the case that the minimal model program eventually contracts $D$.\\

\noindent\textit{Case 2:} The divisor $D$ is never contracted in the MMP and the MMP terminates with a Mori fiber space $\phi:Y_1 \rightarrow Z$ with general fiber of dimension at least 2.\\ 

Here, we have that $\rho(Y_1/Z)=1$. Hence $S_1$ is ample over $Z$, where $S_1$ is the push-forward of $S_Y$ to $Y_1$.
We denote by $B_1$ (resp $M_1$) the push-forward of $B_Y$ (resp $M_Y$) to $Y_1$.
Analogously, we denote by $D_1$
the push-forward of $D_Y$ to $Y_1$.
As $D_1$ is effective it can either be vertical over the base or ample over the base. We argue in each of these cases:\\

\noindent\textit{Case 2a:} The divisor $D_1$ is numerically trivial over the base $Z$.\\

The divisor $D_1$ contains a curve $C$ that gets contracted to a point in $Z$. Indeed, we can take a general point $P\in \phi(D_1)$ and a curve $C\subset \phi ^{-1}(P)$. Therefore $C\cap D_1 \neq \emptyset$, but since $D_1$ is vertical over the base, we have that $C \cdot D_1=0$, hence $C\subset D_1$ is the desired curve.
As $S_1$ is ample over the base $C\cdot S_1 >0$, so $C\cap S_1 \neq \emptyset$, therefore $D_1\cap S_1 \neq \emptyset$. 
Then, we can replace 
$(Y_1,B_1+M_1)$
with a dlt modification,
perform adjunction to a component
$S$ of $\lfloor B_1\rfloor$,
and proceed as in the first case.\\

\noindent\textit{Case 2b:} The divisor $D_1$ is ample over the base $Z$.\\

Let $F$ be the general fiber.
The divisors $D_1|_{F}$ and $S_1|_{F}$ are ample in $F$. As ${\rm dim}(D)|_{F}\geq 1$, we can find a curve $C\subset D_1|_{F}$. Since $S_1|_{F}$ is ample, we have that $C\cdot S_1|_{F}>0$. Therefore $C\cap S_1\neq \emptyset$, so $S_1\cap D_1 \neq \emptyset$. 
Then, we can replace $(Y_1,B_1+M_1)$ 
with a dlt modification,
perform adjunction to a component $S$ of $\lfloor B_1\rfloor$, and proceed as in the first case.\\

\noindent\textit{Case 3:} The divisor $D$ is never contracted in the MMP and the MMP terminates with a Mori fiber space $Y_1 \rightarrow Z$ with general fiber of dimension one.\\

If $D_1$ is vertical over the base, we can proceed as in the case 2a.
Otherwise, we restrict $(Y_1,B_1+M_1)$ to the general fiber  and we obtain a one dimensional generalized log CY pair 
$(\pp^1,B_{\pp^1}+M_{\pp^1})$
so that $d \in {\rm coeff}(B_{\pp^1})$
is contained in $D_d(I)$.
\end{proof}

\begin{remark}\label{remark:reduce-to-dim-c}
{\em 
In the proof of Lemma~\ref{lem:CY-coreg-to-dim}, 
if $c=0$, then we can always take $Z\simeq \pp^1$.
On the other hand, if $c\geq 1$, 
then the projective variety $Z$
can be chosen to have dimension $c$ or $Z\simeq \pp^1$.
}
\end{remark}

\begin{proof}[Proof of Theorem \ref{thm:global-acc-coreg-c}]
Suppose there is no finite set $I_0$. Then there exists an infinite set
$I'\subseteq I$ such that for each $i \in  I'$,
we have that 
\begin{itemize}
    \item the coefficient $i$ is in ${\rm coeff}(B)$ for some generalized pair $(X,B+M)$
    satisfying conditions (1)-(6) of the statement, or
    \item the coefficient $i=\mu_k$ for some generalized pair $(X,B+M)$ satisfying conditions (1)-(6) of the statement.
\end{itemize}
We take an infinite sequence $(d_j)_{j\geq 0}$ of values $d_j\in I'$ and generalized pairs $(X_j,B_j+M_j)$ satisfying the conditions (1)-(6) of the theorem and either $d_j \in {\rm coeff}(B_j)$
or $d_j=\mu_{k,j}$ for some $k,j$. Since $I$ satisfies the DCC, we can take a subsequence of $(d_j)_{j\geq 0}$ that is strictly increasing. By abuse of notation we may also call this sequence $(d_j)_{j\geq 0}$.

Applying Lemma~\ref{lem:CY-coreg-to-dim} to each
generalized pair $(X_j,B_j+M_j)$, with either $d_j\in {\rm coeff}(B_j)$ or $d_j=\mu_{k,j}$, we obtain a sequence of generalized pairs  $(Z_j,B_{Z_j}+M_{Z_j})$ satisfying the following conditions:
\begin{itemize}
    \item the projective variety $Z_j$ has dimension at most $c+1$;
    \item the generalized pair $(Z_j,B_{Z_j}+M_{Z_j})$ is generalized log Calabai-Yau;
    \item the coefficients of $B_{Z_j}$ are in $D(I)$;
    \item  we can write 
    $M'_Z=\sum_j \eta_{k,j} M'_{Z,j}$ where
    $\eta_{k,j}\in I$ and each $M'_{Z,j}$ is a nef Cartier divisor; and
    \item either $d'_j \in {\rm coeff}(B_{Z_j})$
    with $d'_j\in D_{d_j}(I)$ and $d'_j<1$ 
    or 
    $d'_j=d_j=\eta_{k,j}$.
\end{itemize}
By Lemma~\ref{lem:DI-DCC},
we know that $D(I)$ satisfies the descending chain condition.
If we have that $d'_j=d_j=\eta_{k,j}$ for infinitely many $j$, we get a contradiction
of Proposition~\ref{prop:global-acc-dim-n}.
Otherwise, we  assume that each $d'_j$ is of the following form:
\[
d'_j=\frac{N_j-1+m_{j,1}d_j+\sum_{k\geq 2}m_{j,k}d_{j,k}}{N_j}
\]
for some positive integers $N_j,m_{j,1}$
and $m_{j,k}$,
and that $d'_j<1$.
By Proposition~\ref{prop:global-acc-dim-n}, the coefficients $d'_j$ belong to a finite set 
$I_1$, so we may assume they are all equal to a fixed number $d'\in (0,1)$. As $d' \geq 1- \frac{1}{N_j}$, the integer $N_j$ can take only finitely many values. By abuse of notation, we take a subsequence with constant $N_j=N$.
As $m_{j,1}d_j < 1$, we have that \[m_{j,1} < \frac{1}{d_j}\leq \frac{1}{d_1}.\]
Therefore, $m_{j,1}$ can also take finitely many values and we can assume it to be a constant $m$.
Then, we can write:
\[
\sum_{k\geq 2}m_{j,k}d_{j,k}=Nd'-N+1-md_j.
\]
The left-hand side is in $D(I)$, hence it also satisfies the DCC. However, the right-hand side is strictly decreasing, as the only variable is $d_j$, which we chose to be strictly decreasing.
We have a contradiction, so there exists a finite set $I_0$ with the desired property.
\end{proof}

\section{Proof of the theorems}

In this section, we prove the theorems presented in the introduction.
In subsection~\ref{ss:local-acc}, we prove the local ascending chain condition
for log canonical singularities with bounded coregularity.
In subsection~\ref{ss:lct-bounded-coreg}, we prove the ascending chain condition for log canonical thresholds with bounded coregularity.
In subsection~\ref{ss:coreg-zero-one}, we study the log canonical thresholds of coregularity zero and one.
Finally, in subsection~\ref{ss:proofs}, we study the accumulation points of log canonical thresholds with bounded coregularity.

\subsection{Local ascending chain condition}\label{ss:local-acc}

In this subsection, we prove a local version of the ascending chain condition
for 
generalized log canonical singularities with bounded coregularity.

\begin{theorem}\label{thm:local-acc-coreg-c}
Let $c$ be a positive integer.
Let $I$ be a set of nonnegative real numbers satisfying the descending chain condition.
Then, there exists a finite subset $I_0\subseteq I$ 
such that if $(X,B+M)$ satisfies the following conditions:
\begin{enumerate}
    \item the variety $X$ is normal and quasi-projective;
    \item the generalized pair $(X,B+M)$ is generalized log canonical;
    \item the coefficient of any component of $B$ lies in $I$;
    \item we can write $M'=\sum_j \lambda_j M'_j$ where each $\lambda_j\in I$ and each $M'_j$ is a nef Cartier divisor;
    \item the coregularity of $(X,B+M)$ is at most $c$; and
    \item there exists a minimal generalized non-klt center $Z\subset X$ of $(X,B+M)$ 
    which is contained in every component of $B$
    and no $M'_j$ descends over the generic point of $Z$;
\end{enumerate}
then the coefficient of any component of $B$ lie in $I_0$
and each $\lambda_j$ is contained in $I_0$.
\end{theorem}

\begin{proof}
Assume $1\in I$. Let $J\subseteq I$ be the finite set we obtain in Theorem \ref{thm:global-acc-coreg-c}. 
Let 
\[
I_1 = 
\left\{
i\in I \mid \frac{m-1+ki+f}{m}\in J\cap[0,1] \text{ for some }m,k\in \mathbb{Z}_{> 0}, f\in I
\right\},
\] 
which is finite by Lemma~\ref{lem:finite-to-finite}. 
Let $I_0 = I_1\cup \{1\}$.
Without loss of generality, we may assume that $Z$ is a minimal generalized non-klt center.
Let $i\in I$ be the coefficient of a component $D$ of $B$.
We show that $i$ must belong to the finite subset $I_0$.
An analogous argument shows that the $\lambda_j$ belong to $J$.
By assumption, $Z$ is contained in the support of $D$. 
If $Z$ is a divisor on $X$, then $B=Z$ and hence $i=1$.

Suppose that $Z$ is not a divisor on $X$. 
Let $\pi:(Y,B_Y+M_Y)\to (X,B+M)$ be a generalized dlt modification, where 
$B_Y$ is the strict transform of $B$.
We denote by $D_Y$ the strict transform of $D$ on $Y$.
We may choose $\pi$ so that the center of some component $S$
of $\lfloor B_Y\rfloor$ on $X$ is $Z$. \\
 
\noindent\textit{Case 1:} The divisor $D_Y$ intersects $S$ over the generic point of $Z$.\\  

Adjunction formula gives us:
\[
(K_Y+B_Y+M_Y)|_S \sim_\rr K_S+B_S+M_S,
\]
where $(S,B_S+M_S)$ is generalized dlt, the coefficients of $B_S$ belong to $D(I)$, and some component of $B_S$ has a coefficient in $D_i(I)$.
Since $Z$ is a minimal non-klt center of $(X,B+M)$, 
every non-klt center of $(Y,B_Y+M_Y)$ dominates $Z$. 
By adjunction, 
every non-klt center of $(S,B_S+M_S)$ also dominates $Z$. 
Let $(S_z,B_{S_z}+M_{S_z})$ 
be the generalized pair 
induced on a general fiber $S_z$
of the morphism $S\to Z$. 
Then $(S_z,B_{S_z}+M_{S_z})$ is gdlt, $K_{S_z}+B_{S_z}+M_{S_z}\equiv 0$, and some component of $B_{S_z}$ has
a coefficient in $D_i(I)$ (since $B_Y\cap S$ dominates $Z$).
Furthermore, by Lemma \ref{lem:coreg-gen-fiber}, the coregularity of $(S_z,B_{S_z}+M_{S_z})$ is at most $c$. 
By Theorem~\ref{thm:global-acc-coreg-c}, the coefficient of $B_{S_z}$ lies in $J$ and hence $i\in I_0$.\\

\noindent\textit{Case 2:} The divisor $D_Y$ does not intersect $S$ over the generic point of $Z$.\\

Pick $\epsilon>0$ to be sufficiently small and run a $(K_Y+B_Y-\epsilon S+M_Y)$-MMP on $Y$ over $X$, 
with scaling of an ample divisor. 
Since, 
\[
K_Y+B_Y-\epsilon S+M_Y = \pi^*(K_X+B+M) -\epsilon S \equiv_{X} - \epsilon S,
\]
every step of this MMP is $S$-positive.
Furthermore, this MMP terminates with a minimal model $Y_0$ over $X$ by~\cite[Lemma 4.4]{BZ16}. 
The divisor $S$ is not contracted in this MMP (see Claim of the proof of Lemma~\ref{lem:CY-coreg-to-dim}). 
No component of the divisor $B_Y$ is contracted in this MMP since 
their center on $X$ are divisors. 

Let $S_{Y_0}$ be the push-forward of $S$ on $Y_0$.
Analogously, we denote by $B_{Y_0}$ (resp. $M_{Y_0}$ and $D_{Y_0}$)
the push-forward of $B_Y$ (resp. $M_Y$ and $D_Y$) to $Y_0$.
We claim that $S_{Y_0}$ is the set theoretic preimage of $Z$ with respect to the structure morphism $\pi_0: Y_0\rightarrow X$. From our construction, $\pi_0 (S_{Y_0}) = \pi(S) = Z$, so it suffices to prove $\pi_0^{-1}(Z)\subseteq S_{Y_0}$. Suppose there is a component of $\pi_0^{-1}(Z)$ which is not contained in $S_{Y_0}$. Then, we can choose a curve $C\not\subseteq S_{Y_0}$ such that $\pi_0(C)$ is a point in $Z$ and $C\cap S_{Y_0}\neq \varnothing$. However, $-S_{Y_0}$ is nef over $X$ because $(Y_0,B_{Y_0}-\epsilon S_{Y_0}+M_{Y_0})$ is a minimal model and $K_{Y_0}+B_{Y_0}+M_{Y_0}$ is numerically trivial over $X$. 
This is a contradiction, so the claim follows.

We note that the image of $D_{Y_0}$ in $X$ dominates $Z$, 
so $D_{Y_0}\cap S_{Y_0}$ also dominates $Z$ by the above claim. By Lemma~\ref{lem:coreg-mmp}, the coregularity of $(Y_0,B_{Y_0}+M_{Y_0})$ is equal to the coregularity of $(Y,B_Y+M_Y)$.
Thus, we may replace $(Y,B_Y+M_Y)$ by $(Y_0,B_{Y_0}+M_{Y_0})$ 
to assume that $D_Y$ intersects $S$ over the generic point of $Z$, 
$-S$ is nef over $X$, 
and $\pi^{-1}(Z) = S$ holds set-theoretically.
However, we may lose the property that $(Y,B_Y+M_Y)$ is generalized dlt. 

Let $\pi_1\colon (Y_1, B_{Y_1}+M_{Y_1})\to (Y,B_Y+M_Y)$ be a generalized dlt modification.
Let $F=T_1+\dots+T_n$ be the sum of all the $\pi_1$-exceptional divisors.
Let $S_{Y_1}$ be the strict transform of $S$ on $Y_1$.
Write 
\[
\pi_1^* S = S_{Y_1} + \sum_{j=1}^n \lambda_j T_j,
\]
for some $\lambda_j\geq 0$. 
Let $\psi\colon Y_1\rightarrow X$ be the induced
projective birational morphism.
We show that
\[
\psi^{-1}(Z) = S_{Y_1}\cup \bigcup_{\lambda_j>0} T_j,
\]
holds set-theoretically.
We apply a similar argument as before. It suffices to show that 
\[
\psi^{-1}(Z) \subseteq  S_{Y_1}\cup \bigcup_{\lambda_j>0} T_j.
\]
Suppose that the inclusion does not hold. Then, we can find a curve $C$ such that
\begin{itemize}
    \item $C \not \subseteq {\rm supp}(S_{Y_1 }\cup \bigcup_{\lambda_j>0} T_j)$,
    \item $\psi(C)$ is a point in $Z$,
    \item $C$ intersects non-trivially with either $S_{Y_1}$ or some $T_j$ for some $\lambda_j>0$.
\end{itemize}
This implies that
\[
\pi_1^*S \cdot C = \left(S_{Y_1} + \sum_{j=1}^n \lambda_j T_j\right) \cdot C>0,
\]
which contradicts the fact that $-S$ is nef over $X$.
Hence, there is a prime divisor $S_1$ in $\{S_{Y_1}\}\cup \{T_j: \lambda_j>0\}$ 
for which $D_{Y_1}\cap S_1$ dominates $Z$. 
This reduces to the first case. 
Hence, we conclude that $i\in I_0$.
\end{proof}

\subsection{Log canonical thresholds
with bounded coregularity}
\label{ss:lct-bounded-coreg}

In this section, we state 
and we prove the 
ascending chain condition
for generalized log canonical thresholds 
with bounded coregularity.

\begin{theorem}\label{theorem:ACC-coreg}
Let $c$ be a positive integer. 
Let $I$ and $J$ be two sets of nonnegative real numbers satisfying the descending chain condition.
Let ${\rm LCT}_c(I,J)$ be the 
set of thresholds 
$t={\rm lct}(X,B+M;\Gamma+N)$
where the following conditions are satisfied:
\begin{itemize}
    \item $(X,B+M)$ is a generalized log canonical pair;
    \item the coefficients of $B$ lie in $I$
    and 
    the coefficients of $\Gamma$ lie in $J$;
    \item we can write 
    $M'=\sum_j \lambda_j M'_j$ where $\lambda_j\in I$
    and each $M'_j$ is a nef Cartier divisor;
    \item we can write 
    $N'=\sum_j \mu_j N'_j$ where $\mu_j\in J$
    and each $N'_j$ is a nef Cartier divisor; and
    \item the coregularity of
    $(X,B+M+t(\Gamma+N))$ is at most $c$.
\end{itemize}
Then, the set
${\rm LCT}_c(I,J)$ satisfies the ACC.
\end{theorem}

\begin{proof}
Let $t_k:={\rm lct}(X_k,B_k+M_k;\Gamma_k+N_k)$
be a strictly increasing sequence in ${\rm LCT}_c(I,J)$.
For each $k$, let $V_k$ be a minimal 
generalized log canonical center of
$(X_k,B_k+M_k+t_k(\Gamma_k+N_k))$
which is either 
contained in the support of $\Gamma_k$
or in the locus where $N_k$ does not descend.
By Lemma~\ref{lem:coreg-throw-components}, 
we may assume that every component of 
$\Gamma_k$ and $B_k$ contains $V_k$.
Furthermore, we may assume that neither
$N'_{k,j}$ nor $M'_{k,j}$ descend 
at the generic point of $V_k$.
Define 
\[
L = \{i+t_k j \mid i\in I\cup \{0\},j\in J\cup\{0\}, k\in \mathbb{Z}_{>0}\}
\]
so that the coefficients of 
$B_k+t_k\Gamma_k$ lie in $L$
and
the coefficients of 
$M_k+t_kN_k$, in the models where they descend,
lie in $L$.
Then, $L$ satisfies the descending chain condition.
By Theorem~\ref{thm:local-acc-coreg-c}, 
there exists a finite subset $L_0\subseteq L$
such that the coefficients
of $B_k+t_k\Gamma_k$ lie in $L_0$
and the coefficients
of $M_k+t_kN_k$, in the models where they descend,
lie in $L_0$.
Passing to a subsequence, we may
find $i_k\in I$ and $j_k\in J\setminus\{0\}$
with $i_k+t_kj_k=\ell_0$ for some fixed $\ell_0\in L_0$.
This gives a contradiction.
We conclude that the set ${\rm LCT}_c(I,J)$
satisfies the ascending chain condition.
\end{proof}

\subsection{Log canonical thresholds
of coregularity zero and one}
\label{ss:coreg-zero-one}
In this subsection, we study
the set of log canonical thresholds
of coregularity zero
and coregularity one.

\begin{theorem}\label{theorem:explicity-coreg-zero}
Let $I$ and $J$ be two sets of nonnegative real numbers.
We have that 
\[
{\rm LCT}_0(I,J)=
\left\{
\frac{1-i}{j}
\mid 
i\in I^+ \cap [0,1], j\in J^+
\right\}. 
\]
\end{theorem}

\begin{proof}
Denote the set on the right hand side to be $\mathcal{I}$. 
Let $t={\rm lct}(X,B+M;\Gamma+N)$ be a threshold
as in the statement.
Let $V$ be a minimal generalized lc center of $(X,B+M+t(\Gamma+N))$ 
which is either contained in the support of $\Gamma$
or the locus where $N$ does not descend. By Lemma \ref{lem:coreg-throw-components}, we may assume that every component of $\Gamma$ contains $V$
and no $N'_j$ descend over the generic point of $V$. Denote the set
\begin{equation}\label{eq:set-L}
L= \{i+t j \mid i\in I\cup \{0\}, j\in J\cup \{0\}\}
\end{equation}
be the set of all possible coefficient of 
$B+t\Gamma$ and 
coefficients of $M+tN$, where it descends.
Let $\gamma = i+t j$ be the coefficient of a component of $B+t\Gamma$ or $M+tN$ with $j\neq 0$.
By the proof of Theorem \ref{thm:local-acc-coreg-c}, we may find a projective generalized pair $(S,B_S+M_S)$ such that
\begin{itemize}
    \item $(S,B_S+M_S)$ is generalized log canonical,
    \item $K_S+B_S+M_S\equiv 0$,
    \item the coregularity of $(S,B_S+M_S)$ is zero,
    \item every component of $B_S$ has coefficient in $D(L)$,
    \item we can write 
    $M'_S=\sum_j \lambda_j M'_{S,j}$ 
    where $\lambda_j\in L$ and each $M'_{S,j}$ is nef Cartier, and
    \item either component of $B_S$ has coefficient in $D_{\gamma}(L)$ or $\gamma=\lambda_j$ for some $j$.
\end{itemize}
By Lemma \ref{lem:CY-coreg-to-dim} and Remark \ref{remark:reduce-to-dim-c}, we can obtain a generalized pair $(\pp^1,B_{\pp^1}+M_{\pp^1})$ such that 
\begin{enumerate}
    \item the generalized pair $(\pp^1,B_{\pp^1}+M_{\pp^1})$ is generalized log canonical,
    \item $K_{\pp^1}+B_{\pp^1}+M_{\pp^1}\equiv 0$, 
    \item the generalized pair $(\pp^1,B_{\pp^1}+M_{\pp^1})$ has coregularity zero, 
    \item every coefficient of $B_{\pp^1}$ lies in $D(L)$, 
    \item we can write $M_{\pp^1}=\sum_s \lambda_s M_{\pp^1,s}$ where $\lambda_s\in L$
    and each $M'_{\pp^1,s}$ is nef and integral, and
    \item either some component of $B_{\pp^1}$ has coefficient in $D_\gamma(L)$
    or $\gamma=\lambda_j$ for some $j$.
\end{enumerate}
Since $(\pp^1,B_{\pp^1}+M_{\pp^1})$ has coregularity zero, then some coefficient of $B_{\pp^1}$ equals one.
Now, we have that 
\[
{\rm deg}(K_{\pp^1}+B_{\pp^1}+M_{\pp^1})=0.
\]
Expanding this expression, we get an equation of the form
\[
\sum_{k=1}^m \frac{N_k-1 + d_k}{N_k} = 1,
\]
where $d_k\in L^+$ and $N_k\in \mathbb{Z}_{>0}$ and the term one comes from the coefficient-one component of $B_{\pp^1}$. Notice that $1-N_k^{-1}\geq 1/2$ if $N_k \geq 2$ and property (6) implies that $d_k> 0$ for some $k$. Hence we have at most one $k$ for which $N_k \neq 1$. Suppose $N_k = 1$ for all $k\geq 2$. Then
\[
d_1 + N_1(d_2 + \cdots + d_m) = 1.
\]
Write $d_k' = N_1d_k$ if $k\geq 2$ and $d_1'= d_1$. Then $d_k'\in L^+$ for all $k$. We further write $d_k' = i_k + tj_k$ for some $i_k\in I^+$ and $j_k\in J^+$. We have
\[
\sum_{k=1}^m (i_k + tj_k) = 1.
\]
Property (6) ensures that some $j_k\neq 0$ and hence we can solve 
\[
t = \frac{1-\sum_{k=1}^mi_k}{\sum_{k=1}^mj_k}\in \mathcal{I},
\]
as desired. This shows the inclusion
${\rm LCT}_c(I,J)\subseteq \mathcal{I}$ holds.

Now, let $t=(1-i)j^{-1}$ with $i\in I^+ \cap [0,1]$ and $j\in J^+$.
We can write $i=\sum_{k=1}^a m_ki_k$ for $m_k\in \zz_{>0}$
and $i_k\in I$.
We can write $j=\sum_{k=1}^b n_kj_k$ for $n_k\in \zz_{>0}$
and $j_k\in J$.
Let $X = \text{Bl}_q(\text{Bl}_p \pp^2)$ be a sequence of two blow-ups of $\pp^2$, where $p$ is a point on $\pp^2$ and $q$ is a point on the exceptional divisor of $\text{Bl}_p \pp^2\to \pp^2$. Denote the blow-up morphism $X\to \pp^2$ as $\pi$. Let $E_1$ and $E_2$ be the two exceptional divisors of $\pi$, such that the center of $E_1$ is a divisor on $\text{Bl}_p \pp^2$. Let $p_1,\ldots,p_a,p_1',\ldots,p_b'\in E_1\setminus E_2$ be distinct points. Let $q_1,\ldots,q_a,q_1',\ldots,q_b'\in E_2\setminus E_1$ be distinct points. 

For $1\leq k\leq a$, let $C_{k,s}$ and $D_{k,s}$ be pairwise different curves on $X$ such that:
\begin{itemize}
    \item $C_{k,s}\cap E_1 = \{p_k\}$ with intersection multiplicity one and $C_{k,s}\cap E_2 = \varnothing$; and 
    \item $D_{k,s}\cap E_2 = \{q_k\}$ with intersection multiplicity one and $D_{k,s}\cap E_1 = \varnothing$.
\end{itemize}
For $1\leq k\leq b$ and $1\leq s\leq n_k$, let $C_{k,s}'$ and $D_{k,s}'$ be pairwise different curves on $X$ such that:
\begin{itemize}
    \item $C_{k,s}'\cap E_1 = \{p_k'\}$ with intersection multiplicity one and $C_{k,s}'\cap E_2 = \varnothing$; and 
    \item $D_{k,s}'\cap E_2 = \{q_k'\}$ with intersection multiplicity one and $D_{k,s}'\cap E_1 = \varnothing$.
\end{itemize}
Consider the triple
\[
(X,B_X;\Gamma_X) = \left(X, \sum_{k=1}^a\sum_{s=1}^{m_k}i_{k}(C_{k,s}+D_{k,s}); \sum_{k=1}^b\sum_{s=1}^{n_k}j_{k}(C_{k,s}'+D_{k,s}' )\right).
\]
Denote $B = \pi_* B_X$ and $\Gamma = \pi_*\Gamma_X$. Then the coefficients of $B$ lie in $I$ and the coefficients of $\Gamma$ lie in $J$. We can compute
\begin{align*}
   & \pi^*\left(K_{\pp^2} + B + t\Gamma\right) \\&= K_X - E_1 - 2E_2 + B_X + \sum_{k=1}^a i_km_k(2E_1+3E_2) + t\Gamma_X + t\sum_{k=1}^b j_kn_k(2E_1+3E_2) \\
   &=  K_X+B_X+t\Gamma_X + E_1+E_2.
\end{align*}
As a result, $\pi: X\to \pp^2$ is a dlt modification of $(X,B_X+t\Gamma_X)$ and $\text{coreg}(X,B_X+t\Gamma_X) = 0$. Then the triple $(\pp^2,B;\Gamma)$ gives an example with coregularity zero and log canonical threshold $t$.
We conclude that $\mathcal{I}\subseteq {\rm LCT}_0(I,J)$. This finishes the proof of the statement.
\end{proof}

In order to state the theorem for the coregularity one case, 
we need to introduce some notation.
Let $I$ be a set of nonnegative real numbers.
Let $p,q,r$ be three positive integers. 
We set
\[
I^+_{(p,q,r)}:=
\left\{
qri_1+pri_2+pqi_3 \mid 
i_1,i_2,i_3\in I^+
\right\}. 
\]
We define the set 
\[
{\rm LCT}_{1,(p,q,r)}(I,J):=
\left\{
\frac{qr+pr+pq-pqr-i}{j}
\mid i\in I^+_{(p,q,r)} 
\text{ and }
j\in J^+_{(p,q,r)}
\right\}. 
\]
We call this set, the set of {\it weighted log canonical thresholds} of coregularity one.

\begin{theorem}\label{thm:coreg-one}
Let $I$ and $J$ be two sets of nonnegative real numbers.
We have that 
\[
{\rm LCT}_1(I,J):= 
\bigcup_{\frac{1}{p}+\frac{1}{q}+\frac{1}{r}>1} {\rm LCT}_{1,(p,q,r)}(I,J).
\] 
\end{theorem}

\begin{proof}
Let $t\in {\rm LCT}_1(I,J)$.
Proceeding as in the proof of Theorem~\ref{theorem:explicity-coreg-zero}, 
we can produce a generalized log Calabi--Yau pair
$(\pp^1,B_{\pp^1}+M_{\pp^1})$ satisfying the following conditions:
\begin{enumerate}
    \item the coefficients of $B_{\pp^1}$ belong to $D(L)$,
    \item we can write 
    $M_{\pp^1}=\sum_s \lambda_sM_{\pp^1,s}$ where $\lambda_s\in L$ and each $M_{\pp^1,s}$ is a nef integral divisor,  and
    \item either some component of $B_{\pp^1}$ has coefficient in $D_\gamma(L)$ or $\gamma=\lambda_s$ for some $s$.
\end{enumerate}
Above, $L$ is defined as in~\eqref{eq:set-L}
and $\gamma$ is an element of the form $i+tj$ with $j\neq 0$. Expanding the equality
\[
\deg (K_{\mathbb{P}^1} + B_{\pp^1} + M_{\pp^1}) = 0,
\]
we obtain an equation of the form
\[
\sum_{k=1}^m \frac{N_k-1 + d_k}{N_k} = 2,
\]
where $d_k\in L^+$ and $N_k\in \mathbb{Z}_{>0}$ for each $k$. Note that $1-N_k^{-1}\geq 1/2$ if $N_k\geq 2$ and property (3) implies that $d_k> 0$ for some $k$. Thus, there are at most three $k$'s for which $N_k\geq 2$. Assume that $N_k = 1$ for all $k\geq 4$ and $m\geq 3$ (add zero terms if $m\leq 2$). Denote $N_1 = p$, $N_2 = q$, and $N_3 = r$. Then, the previous equation becomes
\[
\frac{1}{p} + \frac{1}{q} + \frac{1}{r} = 1 + \frac{d_1}{p} + \frac{d_2}{q}+\frac{d_3}{r}+\sum_{k=4}^m d_k.
\]
In particular, $1/p+1/q+1/r> 1$.
For each $k$, we write $d_k = i_k + tj_k$ for some $i_k\in I^+$ and $j_k\in J^+$. Then we have
\[
qr+pr+pq-pqr = qr(i_1+tj_1) + pr(i_2+tj_2)+pq(i_3+tj_3)+\sum_{k=4}^m pqr(i_k+tj_k).
\]
Property (3) ensures that some $j_k\neq 0$, so we can solve
\[
t = \frac{qr+pr+pq-pqr-qri_1-pri_2-pqi_3-\sum_{k=4}^m pqri_k}{qrj_1+prj_2+pqj_3+\sum_{k=4}^m pqrj_k} \in I_{(p,q,r)}^+(I,J),
\]
as desired.
We conclude that every element $t \in {\rm LCT}_1(I,J)$ is contained in a set
${\rm LCT}_{1,(p,q,r)}(I,J)$ for some $p,q,r \in \zz_{> 0}$ with $p^{-1}+q^{-1}+r^{-1}>1$.

Now, let $t\in I_{(p,q,r)}^+(I,J)$ for some $p,q,r\in \mathbb{Z}_{>0}$ such that $1/p+1/q+1/r>1$. We can write
\[
t = \frac{qr+pr+pq-pqr-(qri_1+pri_2+pqi_3)}{qrj_1+prj_2+pqj_3}
\]
for some $i_1,i_2,i_3\in I^+$ and $j_1,j_2,j_3\in J^+$. Write
\[
i_k = \sum_{\ell} m_{k,\ell}i_{k,\ell} \text{ and } j_k = \sum_{\ell} n_{k,\ell}j_{k,\ell}
\]
for some $i_{k,\ell}\in I$, $j_{k,\ell}\in J$, and $m_{k,\ell}, n_{k,\ell}\in \mathbb{Z}_{>0}$. 

Let $p_1,p_2$ and $p_3$ be three distinct points on $\pp^1$.
Let $(X;x)$ be the orbifold cone over $\pp^1$ with respect to the $\qq$-divisor $\frac{1}{p}\{p_1\}+\frac{1}{q}\{p_2\}+\frac{1}{r}\{p_3\}$.
Let $\widetilde{X}\rightarrow X$ be the blow-up of $x\in X$. Let $E\subset \widetilde{X}$ be the exceptional divisor, isomorphic to $\pp^1$.
The surface $\widetilde{X}$ has three singularities $x_1,x_2,x_3\in E$.
These singularities are of type $A_p,A_q$ and $A_r$, respectively.
For each $k=1,2,3$ and $\ell$, define:
\begin{itemize}
    \item $\tilde{B}_{k,\ell}$ be an irreducible curve on $\tilde{X}$ such that $\tilde{B}_{k,\ell}|_E = m_{k,\ell} x_k$,  and
    \item $\tilde{\Gamma}_{k,\ell}$ be an irreducible divisor on $\tilde{X}$ such that $\tilde{\Gamma}_{k,\ell}|_E = n_{k,\ell} x_k$.
\end{itemize}
Let
\[
\tilde{B} = \sum_{k,\ell} i_{k,\ell} \tilde{B}_{k,\ell} \text{ and } \tilde{\Gamma} = \sum_{k,\ell} j_{k,\ell}\tilde{\Gamma}_{k,\ell}.
\]
Denote $B= \phi_*\tilde{B}$ and $\Gamma = \phi_*\tilde{\Gamma}.$
Then the coefficients of $B$ lie in $I$ and the coefficients of $\Gamma$ lie in $J$. Note that we have
\[
(K_{\tilde{X}} + \tilde{B} + t\tilde{\Gamma} + E)|_E \sim_{\mathbb{Q}} K_E+\left(\frac{p-1+i_1+tj_1}{p}\right)x_1 +\left(\frac{q-1+i_2+tj_2}{q}\right) x_2 + \left(\frac{r-1+i_3+tj_3}{r}\right)x_3.
\]
Thus, the divisor $K_{\widetilde{X}}+\widetilde{B}+t\widetilde{\Gamma}+E$ intersects $E$ trivially. 
Hence, we have that 
\[
\phi^*(K_X + B +t\Gamma) = K_{\tilde{X}} + \tilde{B} + t\tilde{\Gamma} + E.
\]
Then, the pair $(X,B+t\Gamma;x)$ is log canonical but not klt at the vertex. Furthermore, the pair has coregularity one as $E$ is the only log canonical place.
We conclude that $t\in {\rm LCT}_1(I,J)$. 
This finishes the proof.
\end{proof}

\subsection{Proof of the theorems}
\label{ss:proofs}
In this subsection, we prove the theorems of the introduction.
The first three theorems presented in the introduction are already proved in the previous sections in a broader context.
For the statement about accumulation points, we need a short argument.

\begin{proof}[Proof of Theorem~\ref{introhm:acc-coreg}]
This follows from Theorem~\ref{theorem:ACC-coreg} for generalized pairs.
\end{proof}

\begin{proof}[Proof of Theorem~\ref{introthm:acc-global-coreg}]
This follows from Theorem~\ref{thm:global-acc-coreg-c} for generalized pairs.
\end{proof}

\begin{proof}[Proof of Theorem~\ref{introthm:coreg-zero-one}]
This follows from Theorem~\ref{theorem:explicity-coreg-zero} for generalized pairs.
\end{proof}

In order to prove the last theorem, 
we need to use the following notation.

\begin{notation}
{\em
We denote by
$\mathcal{N}_c(I,t)$
the set of log Calabi--Yau
log canonical pairs $(X,B+\Gamma)$
so that the coefficients of $B$ belong to
$D(I)$, the coefficients of $\Gamma\neq 0$ belong to
$D_t(I)$,
and 
${\rm coreg}(X,B+\Gamma)\leq c$.
We set
\[
N_c(I):=\{t\mid \mathcal{N}_c(I,t) \text{ is non-empty }\}. 
\]
Analogously,
we denote by
$\mathcal{N}_{\dim \leq c}(I,t)$
the set of log Calabi--Yau
log canonical pairs $(X,B+\Gamma)$
so that the coefficients of $B$ belong to $D(I)$,
the coefficients of $\Gamma\neq 0$ belong 
to $D_t(I)$,
and the dimension of $X$ is at most $c$.
We set
\[
N_{\dim\leq c}(I):=\{t\mid \mathcal{N}_{\dim\leq c}(I,t) \text{ is non-empty}\}. 
\]
Note that $N_{\dim \leq c}(I)\subset N_c(I)$
as every log Calabi--Yau pair of dimension at most $c$
has coregularity at most $c$.
}
\end{notation}

\begin{proof}[Proof of Theorem~\ref{introthm:accum-points}]
Let $t\in {\rm LCT}_c(I)$,
by the proof of Theorem~\ref{theorem:ACC-coreg}, we deduce that 
$t\in N_c(I)$. 
By the proof of Theorem~\ref{thm:global-acc-coreg-c}, 
it follows that 
$t\in N_{\dim \leq c}(I)$.
Thus, we conclude that we have an inclusion
\[
{\rm LCT}_c(I)\subseteq N_{\dim \leq c}(I).
\]
As in the proof of Theorem~\ref{thm:coreg-one}, taking the orbifold cone over a log Calabi--Yau pair in $\mathcal{N}_{{\rm dim}\leq c}(I,t)$ gives us the opposite inclusion
\[
N_{\dim \leq c}(I)\subseteq {\rm LCT}_c(I).
\]
Thus, we conclude that 
\begin{equation}\label{eq:lct=n}
{\rm LCT}_c(I) = N_{\dim \leq c}(I).
\end{equation}
Now, we take accumulation points at both sides of the previous equality and get 
\[
{\rm Acc}({\rm LCT}_c(I)) \subseteq {\rm Acc}(N_{\dim \leq c}(I)), 
\]
where ${\rm Acc}(S)$ stands for the 
accumulation points of the set $S$.
By~\cite[Proposition 11.7]{HMX14}, we conclude that 
\[
{\rm Acc}({\rm LCT}_c(I)) = {\rm Acc}(N_{\dim \leq c}(I)) 
\subseteq N_{\dim \leq c-1}(I) = {\rm LCT}_{c-1}(I),
\]
where we used equality~\eqref{eq:lct=n}
in the first and last equalities.
This finishes the proof of the theorem.
\end{proof}

\section{Examples and Questions}

In this section, we give a couple of examples
related to the coregularity of fibrations and singularities.
We also propose some questions for further research.
First, we show an example of log canonical thresholds
on singularities that are closely related to toric singularities, the so-called
complexity one $\mathbb{T}$-singularities
(see, e.g.,~\cite{LLM18,LLM19,LLM20}).

\begin{example}\label{ex:1}
{\em 
Let $(X,B;x)$ be a toric singularity with reduced boundary $B$.
Let $\Gamma$ be a reduced torus invariant divisor through $x$.
Then, we have that 
\[
{\rm lct}(X,B;\Gamma) \in \{0,1\},
\]
holds independently of the dimension of the germ $(X;x)$.

Let $(X,B;x)$ be a complexity one rational log canonical $\mathbb{T}$-singularity with $B$ reduced, i.e.,
$X$ is $n$-dimensional and admits
the effective action of a $(n-1)$-dimensional torus $\mathbb{T}$ and $B$ is $\mathbb{T}$-invariant.
Let $\Gamma$ be an effective
reduced $\mathbb{T}$-invariant divisor.
We can find a $\mathbb{T}$-equivariant projective birational morphism 
$\phi\colon Y\rightarrow X$ so that 
$\phi^{-1}(x)\simeq \pp^1=:C$.
Furthermore, the variety $Y$ admits a good $\mathbb{T}$-quotient to $Z\simeq \pp^1$
and the log pair 
\[
\phi^*(K_X+B+t\Gamma)=K_Y+B_Y+t\Gamma_Y
\] 
has toroidal singularities.
Furthermore, as we are assuming $(X,B+t\Gamma)$ is strictly log canonical at $x$, we may assume that
some log canonical center of $(Y,B_Y+t\Gamma_Y)$ is contained
in $C$.
Hence, we conclude that all the $\mathbb{T}$-invariant divisors of $Y$ which dominate $Z$
must appear with coefficient one in $B_Y+t\Gamma_Y$.
In particular, performing adjunction to $C$,
we obtain a one-dimensional dlt pair on $C\simeq \pp^1$.
Taking the degree of this dlt pair, we either
get a relation of the form 
\[
\left(1-\frac{1}{p}\right)+\left(1-\frac{1}{q}\right)+\left(1-\frac{1}{r}\right)+t=2,
\] 
or
\[
\left(1-\frac{1}{p}+\frac{t}{p}\right)+\left(1-\frac{1}{q}\right)+\left(1-\frac{1}{r}\right)=2,
\]
for certain positive integers $p,q,r$.
We conclude that $t$ must belong to the set 
\[
\left\{ 
\frac{1}{p}+\frac{1}{q}+\frac{1}{r}-1 \mid 
p,q,r\in \zz_{>0}
\right\}\cap [0,1] 
\cup 
\left\{
p\left(\frac{1}{p}+\frac{1}{q}+\frac{1}{r}-1\right) \mid 
p,q,r\in \zz_{>0}
\right\}\cap [0,1]. 
\] 
Note that the previous set of log canonical thresholds $t$
is independent of the dimension of the germ.
Indeed, the coregularity of the previous pairs 
is at most one independently of their dimension.
}
\end{example}

The following example shows that the ascending chain
condition does not hold if we do not bound the coregularity or control the coefficients of the involved divisors.

\begin{example}\label{ex:4}
{\rm 
First, we have ${\rm lct}(\mathbb{A}^1,i\{0\};j \{0\})=(1-i)j^{-1}$.
Then, if either $I$ or $J$ does not satisfy the 
descending chain condition, 
then the set ${\rm LCT}_1(I,J)$ does not satisfy
the ascending chain condition.
Thus, the DCC condition for $I$ and $J$ in the statement of Theorem~\ref{introhm:acc-coreg}
is necessary.

Now, we turn to show that the condition
on the coregularity in Theorem~\ref{introhm:acc-coreg} is also a necessary one.
Let $H_n$ be a smooth hypersurface of degree $n+2$ in $\mathbb{P}^n$.
Then, we have that the pair 
\[
\left(\pp^n, \frac{n+1}{n+2}H_n \right)
\] 
is log Calabi Yau.
Let $\Gamma_n$ be the cone over $H_n$ in $\mathbb{A}^{n+1}$.
Then, $\Gamma_n$ is an irreducible divisor on $\mathbb{A}^{n+1}$
and for every $t$, the log pair
$(\mathbb{A}^{n+1},t\Gamma_n)$ has an isolated singularity at the origin.
We have that
\[
t_n:={\rm lct}(\mathbb{A}^{n+1},\Gamma_n)= 1-\frac{1}{n+2}.
\]
The log pair $(\mathbb{A}^{n+1},t_n\Gamma_n)$ has coregularity $n$.
Indeed, it has a unique log canonical center which corresponds
to the blow-up of the maximal ideal of $\mathbb{A}^{n+1}$ at the origin. 
Hence, we have that
\[
{\rm coreg}(\mathbb{A}^{n+1},t_n\Gamma_n) = n+1 -\dim\mathcal{D}(\mathbb{A}^{n+1},t_n\Gamma_n)-1 = n.
\]
Note that the sequence of $t_n$ violates the ascending chain condition.
However, the coregularity of the sequence
$(\mathbb{A}^{n+1},t_n\Gamma_n)$ is also unbounded. 
}
\end{example}

In the following example,
we show that in a log Calabi--Yau fibration
the coregularity of the general fiber can be arbitrary, 
even if the coregularity of the log pair in the domain is zero.

\begin{example}\label{ex:5}
{\em 
Consider a maximal Fano degeneration $\pi\colon \mathcal{X}\rightarrow \mathbb{A}^1$ of $\pp^n$, i.e., 
\[
\pi^{-1}(\mathbb{A}^1-\{0\}) \simeq \pp^n \times (\mathbb{A}^1-\{0\})
\]
and $\mathcal{X}_0:=\pi^{-1}(\{0\})$ is the union the $n+1$ 
projective toric varieties  intersecting along 
toric subvarieties
and the dual complex of $\mathcal{X}_0$ is a $n$-simplex
(see, e.g.,~\cite[Theorem 1.2]{ML20}).
We consider a general element $H\in |-2(K_{\mathcal{X}}+\mathcal{X}_0)|$, 
then the pair 
\begin{equation}\label{eq:lcy-pair-coreg-ex} 
\left( \mathcal{X}, \mathcal{X}_0+H/2 \right)
\end{equation} 
is log Calabi--Yau over the base and log canonical.
Note that all the log canonical centers of~\eqref{eq:lcy-pair-coreg-ex}
map to zero.
The general fiber of this log pair is 
\[
\left( \pp^n, H|_{\pp^n}/2\right)
\]
which is a log Calabi--Yau klt pair.
Thus, the coregularity of the general fiber is $n$.
However, the log pair~\eqref{eq:lcy-pair-coreg-ex}
has coregularity zero as it is dlt and 
it has a zero-dimensional log canonical center contained in $\mathcal{X}_0$.
Note that this example is over an affine base, however, it can be turned into a projective example by compactifying at infinity as a product.
}
\end{example}

In the following example, 
we show examples of germs with arbitrary dimension
and bounded coregularity.

\begin{example}\label{ex:3}
{\em 
For examples of log Calabi--Yau pairs with coregularity zero 
and arbitrary dimension, 
we can simply consider a log Calabi--Yau toric pair $(T,\Gamma_T)$. 
In the local case, we can consider toric singularities
with the reduced torus invariant boundary.
In these examples, the number of components of the boundary tends to be high.
Indeed, we have exactly 
$\dim T + \rank{\rm Cl}_\qq(T)$ prime components in the boundary.

We show that, in arbitrary dimension, there are examples of log canonical singularities of coregularity zero 
with a unique prime component in the boundary and coefficient $\frac{1}{2}$.
Let $(\mathcal{X},\mathcal{X}_0+H_1/2)$
and $(\mathcal{X},\mathcal{X}_0+H_2/2)$ be two log Calabi--Yaus constructed as in Example~\ref{ex:5}. 
Here, $H_1$ and $H_2$ are two different general elements of the linear system 
$|-2(K_{\mathcal{X}}+\mathcal{X}_0)|$.
We let $\pi\colon\mathcal{Y}\rightarrow \mathcal{X}$ be the blow-up of 
the intersection of $H_1$ with a component of $\mathcal{X}_0$.
We let $E$ be the exceptional divisor of this projective birational morphism.
We denote by $F_1$ (resp. $F_2$) the
strict transform of $H_1$ (resp. $H_2$) on $\mathcal{Y}$,
and by $\mathcal{Y}_0$
the strict transform of $\mathcal{X}_0$ on $\mathcal{Y}$.
Note that we have 
\begin{equation}\label{ex:log-crepant-1}
\pi^*(K_{\mathcal{X}}+\mathcal{X}_0+H_1/2)=
K_{\mathcal{Y}}+\mathcal{Y}_0+F_1/2+E/2.
\end{equation}
and
\begin{equation}\label{ex:log-crepant-2}
\pi^*(K_{\mathcal{X}}+\mathcal{X}_0+H_2/2)=
K_{\mathcal{Y}}+\mathcal{Y}_0+F_2/2.
\end{equation} 
Due to~\eqref{ex:log-crepant-1}, we know that the morphism
$\mathcal{Y}\rightarrow \mathbb{A}^1$ is of Fano type.
Observe that 
\begin{equation}\label{ex:log-div-5} 
 K_{\mathcal{Y}}+\mathcal{Y}_0+F_1/2+(1-\epsilon)E/2,
\end{equation} 
is dlt and pseudo-effective over the base. 
We run a minimal model program for~\eqref{ex:log-div-5} over the base 
$\mathbb{A}^1$. 
Since $\mathcal{Y}\rightarrow \mathbb{A}^1$ is of Fano type,
this minimal model program terminates with a good minimal model over the base.
We call this model $\mathcal{Z}\rightarrow \mathbb{A}^1$.
Let $E_{\mathcal{Z}}$ be the push-forward of $E$ on $\mathcal{Z}$.
Analogously, 
let $F_{\mathcal{Z},i}$ be the push-forward of $F_i$ on $\mathcal{Z}$.
Then, we conclude that the log pair 
\[
\left(\mathcal{Z},F_{\mathcal{Z},2}/2\right)
\] 
is log crepant to the pair~\eqref{ex:log-crepant-2}.
Thus, it has coregularity zero.
Furthermore, this log pair has a unique irreducible component of coefficient $\frac{1}{2}$.
Let $X$ be the cone of $-K_{\mathcal{Z}}$ over $\mathbb{A}^1$
and $\Gamma$ be the cone over $F_{\mathcal{Z},1}$.
We denote by $x\in X$ the point corresponding to 
$\{0\}$ in this cone
with one-dimensional fixed point locus.
Then, we have that 
$(X;x)$ is a $n$-dimensional $\qq$-factorial klt singularity.
On the other hand, 
the pair 
$(X,\Gamma/2;x)$ is a $n$-dimensional $\qq$-factorial pair,
$\Gamma$ is prime, 
and the coregularity of $(X,\Gamma/2;x)$ is zero.
}
\end{example}

A polytopal variant of the log canonical threshold,
the so-called log canonical threshold polytopes, 
were introduced by Musta\c{t}\u{a} and Libgober in~\cite{ML11}.
Given a projective variety $X$
and a sequence of $\qq$-Cartier effective divisors 
$D_1,\dots,D_s$, we can define the 
{\em log canonical threshold polytope}, to be
\[
{\rm LCT}(X;D_1,\dots,D_r):= 
\left\{ 
(t_1,\dots,t_s)\in \rr^s \mid 
(X,t_1D_1+\dots+t_sD_s) \text{ is log canonical}
\right\}. 
\]
The ACC for log canonical threshold polytopes of the same dimension and number of divisors was proved in~\cite{HLQ21}.
This means that there is no infinite sequence of these polytopes
so each of them strictly contains the previous one.
Instead, we can consider log canonical threshold polytopes
on which the coregularity is bounded when the thresholds are computed. 
Analogously to the main theorem of this article, we can ask:

\begin{question}
Does the ascending chain condition for log canonical threshold polytopes 
with bounded coregularity hold? 
\end{question}

The study of numerical thresholds 
naturally leads to the study of pseudo-effective thresholds, i.e.,
questions related to the Fujita's spectrum conjecture (see, e.g.,~\cite{HL20,HL21}).
Given a log canonical pair $(X,\Delta)$ and an effective divisor
$\Gamma$ on $X$, we define the {\em pseudo-effective threshold}
of $(X,\Delta)$ with respect to $\Gamma$, denoted by
$p(X,\Delta;\Gamma)$ to be the smallest positive real number $p$ for which
$K_X+\Delta+p\Gamma$ is pseudo-effective.
We set $p=\infty$ if the previous divisor is never pseudo-effective
and zero if $K_X+\Delta$ is already pseudo-effectve.
We can define: 
\[
\mathcal{P}_c(I,J):=\left\{ 
p \mid p=p(X,\Delta;\Gamma), {\rm coeff}(\Delta)\in I, 
{\rm coeff}(\Gamma)\in J, \text{ and }
{\rm coreg}(X,\Delta+p\Gamma)\leq c
\right\}. 
\]
Then, it is natural to ask:

\begin{question}
Let $c$ be a positive integer.
Let $I$ and $J$ be two sets of nonnegative real numbers
satisfying the descending chain condition.
Does $\mathcal{P}_c(I,J)$ satisfy the ascending chain condition?
\end{question}

We expect both previous questions to have positive answers. 
However, some new ideas must be introduced in order to prove them.

\bibliographystyle{habbvr}
\bibliography{mybibfile}

\vspace{0.5cm}
\end{document}